\documentclass[a4paper, 12pt]{article}

\let\Ho\H %

\usepackage[utf8]{inputenc}
\usepackage[T1]{fontenc}
\usepackage[a4paper, margin=2.5cm]{geometry}
\usepackage{needspace}
\usepackage{libertine} % text font
\usepackage{inconsolata} % texttt font

\usepackage{amsmath, amsthm, amssymb}
\usepackage{graphicx}
\usepackage{enumerate}
\usepackage[shortlabels]{enumitem}
\usepackage[textsize=small, textwidth=2cm, color=yellow]{todonotes}
\usepackage{thmtools,mathtools}
\usepackage{thm-restate}
\usepackage{authblk}
\definecolor{CornflowerBlue}{rgb}{0.39, 0.58, 0.93}
\definecolor{DarkGoldenrod}{rgb}{0.72, 0.53, 0.04}
\definecolor{BritishRacingGreen}{rgb}{0.0, 0.26, 0.15}
\definecolor{DarkMagenta}{rgb}{0.55, 0.0, 0.55}
\definecolor{AO}{rgb}{0.0, 0.5, 0.0}
\definecolor{BostonUniversityRed}{rgb}{0.8, 0.0, 0.0}
\definecolor{myRed}{rgb}{0.8, 0.0, 0.0}
\definecolor{DarkMidnightBlue}{rgb}{0.0, 0.2, 0.4}
\definecolor{DarkTangerine}{rgb}{1.0, 0.66, 0.07}
\definecolor{AppleGreen}{rgb}{0.55, 0.71, 0.0}
\definecolor{BrightUbe}{rgb}{0.82, 0.62, 0.91}
\definecolor{Amethyst}{rgb}{0.6, 0.4, 0.8}
\definecolor{DarkGray}{rgb}{0.52, 0.52, 0.51}
\definecolor{Gray}{rgb}{0.66, 0.66, 0.66}
\definecolor{BananaYellow}{rgb}{1.0, 0.88, 0.21}
\definecolor{Amber}{rgb}{1.0, 0.75, 0.0}
\definecolor{LightGray}{rgb}{0.83, 0.83, 0.83}
\definecolor{PrincetonOrange}{rgb}{1.0, 0.56, 0.0}
\definecolor{DeepCarrotOrange}{rgb}{0.91, 0.41, 0.17}
\definecolor{bordeaux}{RGB}{100,0,50}

\definecolor{cerise}{rgb}{0.93, 0.23, 0.51}

\usepackage{hyperref}

\hypersetup{
    colorlinks=true,
    linkcolor = bordeaux,
    citecolor = bordeaux,
    urlcolor = bordeaux,
    bookmarksopen=true,
    bookmarksnumbered,
    bookmarksopenlevel=2,
    bookmarksdepth=3
}

\usepackage{float}
\usepackage[normalem]{ulem}
\usepackage{xcolor}
\usepackage[ruled,vlined]{algorithm2e}

\usepackage{comment}
\usepackage{placeins}
\usepackage{pdfpages}

\usepackage{pgf}
\usepackage{pgffor}
\usepackage{xspace}
\usepackage{tikz}
\usepackage{tkz-graph,tkz-berge}
\usetikzlibrary{fit}
\usetikzlibrary{arrows}
\usetikzlibrary{patterns}
\usetikzlibrary{calc}
\usetikzlibrary{shapes}
\usetikzlibrary{positioning}
\usetikzlibrary{math}
\usetikzlibrary{backgrounds}

\renewenvironment{abstract}
{\small\vspace{-1em}
\begin{center}
\bfseries\abstractname\vspace{-.5em}\vspace{0pt}
\end{center}
\list{}{
\setlength{\leftmargin}{0.6in}%
\setlength{\rightmargin}{\leftmargin}}%
\item\relax}
{\endlist}

\declaretheorem[name=Lemma, numberwithin = section]{lemma}
\declaretheorem[name=Theorem, sibling=lemma]{theorem}

\declaretheorem[name=Definition, sibling=lemma]{definition}
\declaretheorem[name=Corollary, sibling=lemma]{corollary}
\declaretheorem[name=Conjecture, sibling=lemma]{conjecture}

\declaretheorem[name=Observation, sibling=lemma]{observation}

\declaretheorem[name=Scenario, sibling=lemma]{scenario}

\def\cqedsymbol{\ifmmode$\lrcorner$\else{\unskip\nobreak\hfil
\penalty50\hskip1em\null\nobreak\hfil$\lrcorner$
\parfillskip=0pt\finalhyphendemerits=0\endgraf}\fi}

\newcommand{\say}[1]{``#1''} %place quotes around something

\def\D{\mathcal{D}} % dominating set
\def\dichi{\overrightarrow{\chi}} %directed chromatic number

\def\stmt#1{\vspace{0cm}\begin{equation}\longbox{#1}\end{equation}}

\makeatletter
\newcommand{\leqnomode}{\tagsleft@true}

\newcommand{\reqnomode}{\tagsleft@false}
\makeatother
\def\longbox#1{\parbox{0.85\textwidth}{\textit{#1}}}

\definecolor{nblue}{rgb}{0.38, 0.51, 0.71} %glaucous, 97,130,181, #6182B5
\definecolor{darkblue}{RGB}{17, 42, 60} % 112A3C
\definecolor{nred}{RGB}{175, 49, 39} % AF3127

\definecolor{norange}{RGB}{217, 156, 55} % D99C37
\definecolor{ngreen}{RGB}{144, 169, 84} % 90A954
\definecolor{palegreen}{RGB}{197, 184, 104} % C5B868

\definecolor{nyellow}{RGB}{250, 199, 100} % FAC764
\definecolor{brokenwhite}{RGB}{218, 192, 166} % DAC0A6
\definecolor{brokengrey}{rgb}{0.77, 0.76, 0.82} % {196,194,209}, C4C2D1
\usepackage[absolute]{textpos}
\usepackage{mathdots}
\usepackage{soul}
\definecolor{grun}{rgb}{0, 0.5, 0.5}
\definecolor{violet}{RGB}{177, 0.5, 255}

\newcommand{\ora}[1]{\overrightarrow{#1}}
\newcommand{\olra}[1]{\overleftrightarrow{#1}}

\newcommand{\pathminimizingclosedtournament}{\hyperref[def:closedtournament]{path-minimizing closed tournament}}

\newcommand{\forwardinduced}{\hyperlink{def:forwardinduced}{forward-induced }}

\title{Proving a directed analogue of the Gyárfás-Sumner conjecture for orientations of $P_4$ %
\thanks{\parbox[t]{0.8\linewidth}{TM, MP, and US 
 received funding from the European Research Council (ERC) under the European Union's Horizon 2020 research and innovation programme Grant Agreement 714704. LC was supported the Institute for Basic Science (IBS-R029-C1). AR was supported by the ANR project Digraphs (ANR-19- CE48-0013-01) and the Institute for Basic Science (IBS-R029-C1).}%
\ \ \parbox[t]{0.15\linewidth}{~\\[-3mm]\includegraphics[width=30px]{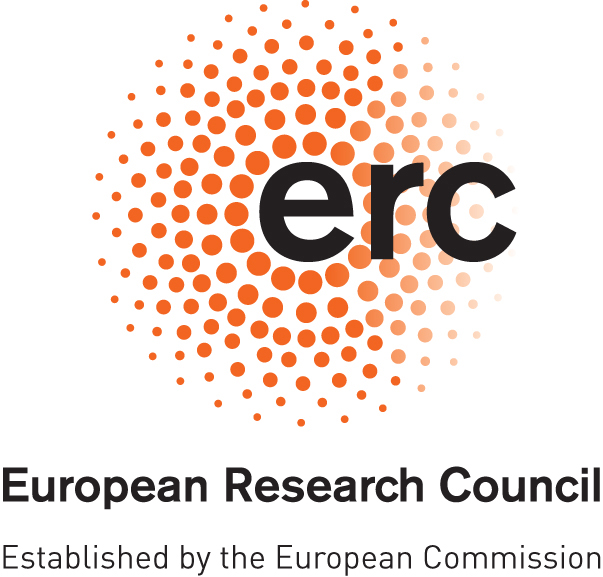}\\\includegraphics[width=30px]{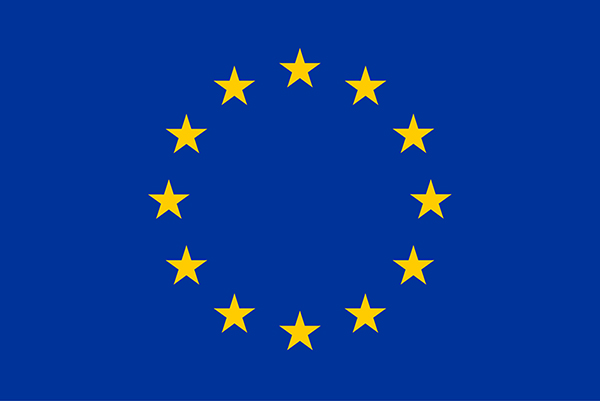}}
}}

\author[1]{Linda Cook}
\author[2]{Tomáš Masařík}
\author[2]{Marcin Pilipczuk}
\author[1,3,4]{\\Amadeus Reinald}
\author[2,5]{Uéverton~S.~Souza}
\affil[1]{Institute for Basic Science, Daejeon, Republic of Korea}
\affil[2]{Institute of Informatics, Faculty of Mathematics, Informatics and Mechanics, University of Warsaw, Warsaw, Poland}
\affil[3]{École Normale Supérieure de Lyon, Lyon, France}
\affil[4]{Université Côte d'Azur, Inria, CNRS, I3S, Sophia Antipolis, France}
\affil[5]{Instituto de Computação, Universidade Federal Fluminense, Niterói, Brasil}

\date{}

\begin{document}

\maketitle

\begin{abstract}
An oriented graph is a digraph that does not contain a directed cycle of length two.
An (oriented) graph $D$ is \emph{$H$-free} if $D$ does not contain $H$ as an induced sub(di)graph.
The Gyárfás-Sumner conjecture is a widely-open conjecture on simple graphs, which states that for any forest $F$, there is some function $f$ such that every $F$-free graph $G$ with clique number $\omega(G)$ has chromatic number at most $f(\omega(G))$.
Aboulker, Charbit, and Naserasr [Extension of Gyárfás-Sumner Conjecture to Digraphs; E-JC 2021] proposed an analogue of this conjecture to the dichromatic number of oriented graphs.
The \emph{dichromatic number} of a digraph $D$ is the minimum number of colors required to color the vertex set of $D$ so that no directed cycle in $D$ is monochromatic.

Aboulker, Charbit, and Naserasr's $\dichi$-boundedness conjecture states that for every oriented forest $F$, there is some function $f$ such that every $F$-free oriented graph $D$ has dichromatic number at most $f(\omega(D))$, where $\omega(D)$ is the size of a maximum clique in the graph underlying $D$.
In this paper, we perform the first step towards proving Aboulker, Charbit, and Naserasr's $\dichi$-boundedness conjecture by showing that it holds when $F$ is any orientation of a path on four vertices.
\end{abstract}

\section{Introduction}

In a simple graph, the size of a maximum clique gives a lower bound on its chromatic number.
But if a graph contains no large cliques, does it necessarily have small chromatic number?
This question has been answered in the negative.
In the mid-twentieth century, Mycielski \cite{Mycielski1955} and Zykov \cite{zykov1949} gave constructions for triangle-free graphs with arbitrarily large chromatic number.
Hence we may ask the following question instead: Given some fixed graph $H$, do graphs with a bounded clique number that do not contain $H$ \emph{as an induced subgraph} have bounded chromatic number?
In 1959, Erd\Ho{o}s showed that there exist graphs with arbitrarily high girth and arbitrarily high chromatic number \cite{erdosBook}.
Hence, the answer to the previous question is \say{no} whenever $H$ contains a cycle, and thus we need only consider the question when $H$ is a forest.
Around the 1980s, Gyárfás and Sumner independently conjectured \cite{gyarfasConjecture, sumnerConjecture} that for any forest $H$, all graphs with bounded clique number and no induced copy of $H$ have bounded chromatic number.
The conjecture has been proven for some specific classes of forests but remains largely open; see \cite{ScottSeymourSurveyChiBddd} for a survey of related results. 
This paper concerns an extension of the Gyárfás-Sumner conjecture to directed graphs proposed by Aboulker, Charbit, and Naserasr \cite{AlboukerGyarfasSumner4Digraphs2020}. We will state the Gyárfás-Sumner conjecture and its extension to directed graphs more formally after introducing some necessary terminology.

A directed graph, or \textit{digraph}, is a pair $D=(V,E)$ where $V$ is the vertex set and $E$ is a set of ordered pairs of vertices in $V$ called the arc set.
We call a digraph \emph{oriented} if it has no \emph{digon} (directed cycle of length two).
This paper will focus on finite, simple, oriented graphs.

For a digraph $D = (V, E)$ we define the \emph{underlying graph} of $D$ to be the graph $D^* = (V, E^*)$ where $E^*$ is the set obtained from $E$ by replacing each arc $e \in E$ by an undirected edge between the same two vertices.
We say two vertices in $D$ are \emph{adjacent} or \emph{neighbors} if they are adjacent in $D^*$.
If $(v,w)$ is an arc of $D$ we say say that $v$ is an \emph{in-neighbor} of $w$ and that $w$ is an \emph{out-neighbor} of $v$. 
We denote the set of neighbors of a vertex $v \in V(D)$ by $N(v)$ and we denote $N(v) \cup \{v\}$ by $N[v]$.
For a set of vertices $S \subseteq V(D)$ we let $N(S)$ and $N[S]$ denote the sets $\cup_{v \in S} N(v) \setminus S$ and $\cup_{v \in S} N[v]$.
We call $N(S)$ the \emph{neighborhood} of $S$ and $N[S]$ the \emph{closed neighborhood} of $S$.
For a subdigraph $H \subseteq D$ we let $N(H)$ denote the set $N(V(H))$.

We let $P_t$ denote the path on $t$ vertices. We say an oriented path is a \emph{directed path} if its vertices are $p_1, p_2, \dots, p_t$, in order, and its orientation is $p_1 \rightarrow p_2 \rightarrow \dots \rightarrow p_t$. We let $\ora{P_t}$ denote the directed path on $t$ vertices. We say a digraph $D$ is \emph{strongly connected} if for every $v, w \in V(D)$ there is a directed path starting at $v$ and ending at $w$. An induced subdigraph $H$ of a digraph $D$ is a \emph{strongly connected component} of $D$ if it is strongly connected and every induced subgraph $H'$ of $D$ such that $H \subseteq H'$ is not strongly connected. We call a strongly connected $H$ component a source (sink) component of $D$ if every arc between $V(H)$ and $V(D \setminus H)$ begins (ends) in $V(H)$. 

We call a digraph whose underlying graph is a clique a \emph{tournament}.
As we consider only oriented digraphs, this definition corresponds to the standard definition of a tournament in the literature.
Given a (di)graph $G$ and $S \subseteq V$, we denote the sub(di)graph of $G$ induced by $S$ as $G[S]$.
We say that a (di)graph $G$ \textit{contains} a (di)graph $H$ if $G$ contains $H$ as an induced sub(di)graph.
If $G$ does not contain a (di)graph $H$ we say that $G$ is \emph{$H$-free}.
If $G$ does not contain any of the (di)graphs $H_1, H_2, \dots, H_k$ we say $G$ is \emph{$(H_1, H_2, \dots, H_k)$-free}.
The \emph{clique number} and the \emph{chromatic number} of a digraph are the chromatic number and clique number of its underlying graph, respectively.
We denote the clique number and the chromatic number of a (di)graph $G$ by $\omega(G)$ and $\chi(G)$, respectively.
We say that a graph $H$ is \emph{$\chi$-bounding} if there exists a function $f$ with the property that every $H$-free graph $G$ satisfies $\chi(G) \leq f(\omega(G))$.
In this language, \cite{erdosBook} implies all $\chi$-bounding graphs are forests.
We are now ready to state the Gyárfás-Sumner conjecture more formally.

\begin{conjecture}[The Gyárfás-Sumner conjecture \cite{gyarfasConjecture, sumnerConjecture}]
\label{con:gs}
Every forest is $\chi$-bounding.
\end{conjecture}
Today, the conjecture is only known to hold for restricted classes of forests.
For example, Gyárfás showed that it holds for paths \cite{Gyarfas1987} via a short and elegant proof.
Subsequently, the conjecture was proven for other classes of forests. For example, the following classes of trees have been proven to be $\chi$-bounding:
\begin{itemize}
    \item Trees of  radius two by Kierstead and Penrice in 1994 \cite{kiersteadPenrice1994},
    \item Trees that can be obtained from a tree of radius two by subdividing \textit{every} edge incident to the root exactly once by Kierstead and Zhu in 2004 \cite{kiersteadZhu}, and
    \item Trees that can be obtained from a tree of radius two by subdividing \textit{some of the} edges incident to the root exactly once by Scott and Seymour in 2020 \cite{scottNewBrooms2020}.
\end{itemize}
Note that the class of trees described by the third bullet contains the classes described in both the first and second bullet.
See the survey of Scott and Seymour \cite{ScottSeymourSurveyChiBddd} for an overview of the state of the conjecture from 2020.

How can the \hyperref[con:gs]{Gyárfás-Sumner conjecture} be adapted to the directed setting?
A first idea is to call an oriented graph $H$ \emph{$\chi$-bounding} if there exists a function $f$ with the property that every $H$-free oriented graph $D$ satisfies $\chi(D) \leq f(\omega(D))$.
Then, once again, by \cite{erdosBook}, all $\chi$-bounding oriented graphs are oriented forests.
Note that if an oriented graph $H$ is $\chi$-bounding, its underlying graph $H^*$ is also $\chi$-bounding.
However, the converse does not hold, as, for instance, $P_4$ is $\chi$-bounding, but there exist orientations of $P_4$ that are not $\chi$-bounding.
There are four different orientations of $P_4$, up to reversing the order of the vertices on the whole path: $$\rightarrow \rightarrow \rightarrow , \rightarrow \leftarrow \rightarrow , \rightarrow \leftarrow \leftarrow, \leftarrow \leftarrow \rightarrow$$
Only the last two oriented graphs in the list are $\chi$-bounding:

\begin{itemize}
  \item \hypertarget{def:p4}{Recall, we denote the oriented $P_4$ with orientation $\rightarrow \rightarrow \rightarrow$ by $\hyperlink{def:p4}{\ora{P_4}}$.} In 1991, Kierstead and Trotter \cite{kiersteadtrotter1991}, showed that $\hyperlink{def:p4}{\ora{P_4}}$ is not $\chi$-bounding. Their construction was inspired by Zykov's construction of triangle-free graphs with a high chromatic number \cite{zykov1949}, and builds $\hyperlink{def:p4}{\ora{P_4}}$-free oriented graphs with arbitrarily large chromatic number and no clique of size three.
  \item \hypertarget{def:a4}{Around 1990, Gyárfás pointed out that $\leftarrow \rightarrow \leftarrow$ is not $\chi$-bounding, as witnessed by an orientation of the shift graphs on pairs \cite{gyarfas}. We will denote the $P_4$ with orientation $\leftarrow \rightarrow \leftarrow$ by $\hyperlink{def:a4}{\ora{A_4}}$.}
  \item \hypertarget{def:q4}{Chudnovsky, Scott and Seymour \cite{chudnovsky2019orientations} showed that $\rightarrow \leftarrow \leftarrow$ and $\leftarrow \leftarrow \rightarrow$ are both $\chi$-bound\-ing in 2019. In the same article, the authors show that orientations of stars are also $\chi$-bounding (stars are the class of complete bipartite graphs $K_{1,t}$ for any $t \geq 1)$. We will denote  $\rightarrow \leftarrow \leftarrow$ and $\leftarrow \leftarrow \rightarrow$ by $\hyperlink{def:q4}{\ora{Q_4}}$ and $\hyperlink{def:q4}{\ora{Q_4}'}$, respectively.}
\end{itemize}

Our first attempt at adapting the \hyperref[con:gs]{Gyárfás-Sumner conjecture} to oriented graphs failed for oriented paths such as $\hyperlink{def:p4}{\ora{P_4}}$ and  $\hyperlink{def:a4}{\ora{A_4}}$.
Hence, we focus on a different approach proposed by Aboulker, Charbit, and Naserasr \cite{AlboukerGyarfasSumner4Digraphs2020} which uses a concept called \say{dichromatic number}.
Directed coloring, or dicoloring, is a weakening of coloring defined on digraphs and was proposed by Neumann-Lara and subsequently developed by Erd\Ho{o}s and Neumann-Lara \cite{erdosmentionsdichi,NeumannLaraIntroducesDiChi1982}.
A \emph{dicoloring} of a digraph $D$ is a partition of $V(D)$ into classes, or colors, such that each class induces an acyclic digraph (that is, there is no monochromatic directed cycle).
The \emph{dichromatic number} of $D$, denoted as $\dichi(D)$, is the minimum number of colors needed for a dicoloring of $D$.
Notice that every coloring of a directed graph $D$ is also a dicoloring, thus $\dichi(D) \leq \chi(D)$.
We say a class of digraphs $\mathcal{D}$ is \emph{$\dichi$-bounded} if there exists a function $f$ such that every $D \in \mathcal{D}$ satisfies $\dichi(D) \leq f(\omega(D))$ and we call such an $f$ a \emph{$\dichi$-binding function} for $\mathcal{D}$.
We say that a digraph $H$ is $\dichi$-\emph{bounding} if the class of $H$-free oriented graphs is $\dichi$-bounded.

We can now state Aboulker, Charbit, and Naserasr's dichromatic analogue to the  \hyperref[con:gs]{Gyárfás–Sumner conjecture} for digraphs.
For brevity, we will call this conjecture the \say{ACN $\dichi$-boundedness} conjecture in the remainder of this paper.
Note, the ACN $\dichi$-boundedness conjecture was originally published as Conjecture 4.4 in \cite{AlboukerGyarfasSumner4Digraphs2020}.

\begin{conjecture}[The ACN $\protect\dichi$-boundedness conjecture~\cite{AlboukerGyarfasSumner4Digraphs2020}]\label{con:AlboukerGyarfasSumner4Digraphs2020}
Every oriented forest is $\dichi$-bound\-ing.
\end{conjecture}

The converse of the \hyperref[con:AlboukerGyarfasSumner4Digraphs2020]{ACN $\dichi$-boundedness conjecture} holds; all $\dichi$-bounding digraphs must be oriented forests. Indeed, Harutyunyan and Mohar proved that there exist oriented graphs of arbitrarily large undirected girth and dichromatic number \cite{HARUTYUNYAN_Mohar_2012}. Oriented graphs of sufficiently large undirected girth (and no digon) forbid any fixed digraph that is not an oriented forest.
Hence, no digraph containing a digon or a cycle in its underlying graph is $\dichi$-bounding. Moreover, for any finite list of digraphs $D_1, D_2, \dots, D_k$, if the class of $(D_1, D_2, \dots, D_k)$-free oriented graphs is $\dichi$-bounded then one of $D_1$, $D_2, \dots, D_k$ must be a forest.
One might ask whether the situation changes when we forbid an infinite list of oriented graphs.
We list some results related to this:
\begin{itemize}
  \item In \cite{sophieandfriends}, Carbonero, Hompe, Moore, and Spirkl provided a construction for oriented graphs with clique number at most three, arbitrarily high dichromatic number, and no induced directed cycles of odd length at least 5. They use this construction to disprove a well-known stronger version of the  \hyperref[con:gs]{Gyárfás-Sumner conjecture} sometimes referred to as \say{Esperet's conjecture} (see also \cite{ScottSeymourSurveyChiBddd}).
  
  \item In \cite{AboulkerChordalDirectedGraphsAreNotChiBounded2022}, Aboulker, Bousquet, and de Verclos showed that the class of chordal oriented graphs, that is, oriented graphs forbidding induced directed cycles of length greater than three, is not $\dichi$-bounded, answering a question posed in \cite{sophieandfriends}.
  
  \item In \cite{CarboneroDirectedKholesnotChiBounded2022}, Carbonero, Hompe, Moore, and Spirkl extended the result of \cite{sophieandfriends} to $t$-chordal graphs. A digraph is \emph{$t$-chordal} if it does not contain an induced directed cycle of length other than $t$. In \cite{CarboneroDirectedKholesnotChiBounded2022} the authors showed that $t$-chordal graphs are not $\dichi$-bounded, but $t$-chordal $\ora{P_t}$-free graphs are $\dichi$-bounded.
\end{itemize}

Note that Conjecture~\ref{con:AlboukerGyarfasSumner4Digraphs2020} only considers \emph{oriented} graphs.
This is the only sensible case.
By the result of Harutyunyan and Mohar \cite{HARUTYUNYAN_Mohar_2012} if $F$ contains a digon, then the class of $F$-free \emph{oriented} graphs is not $\dichi$-bounded.
If $F$ contains no digons and at least one edge, then the class of $F$-free \emph{digraphs} is not $\dichi$-bounded;
Any  digraph obtained from a graph by replacing every edge with a digon does not contain any oriented graph with at least one edge as an induced subgraph. 
Hence by \cite{Mycielski1955, zykov1949}, for any choice of an oriented graph with at least one edge $F$, there exist $F$-free digraphs (with digons) that have arbitrarily high dichromatic number and do not contain a triangle in their underlying graph. 

The \hyperref[con:AlboukerGyarfasSumner4Digraphs2020]{ACN $\dichi$-boundedness conjecture} is still widely open.
It is not known whether the conjecture holds for any orientation of any tree $T$ on at least five vertices that is not a star.
In particular, it is not known whether the conjecture holds for oriented paths.
In contrast, Gyárfás showed that every path is $\chi$-bounding in the 1980's \cite{gyarfasConjecture,Gyarfas1987}.
We will introduce some terminology before discussing the status of the \hyperref[con:AlboukerGyarfasSumner4Digraphs2020]{ACN $\dichi$-boundedness conjecture} for oriented paths in more detail.
For $t \leq 3$, $P_t$ is $\dichi$-bounding.
(This can be proven by, for example, noting that for $t \geq 3$ the graph $P_t$ is a star and applying Chudnovsky, Scott, and Seymour's result \cite{chudnovsky2019orientations} that every orientation of a star is $\chi$-bounding and therefore also $\dichi$-bounding.)
However, for $t \geq 4$, the picture gets more complicated:
\begin{itemize}
    \item Let $T$ be any fixed orientation of $K_3$. In \cite{AlboukerGyarfasSumner4Digraphs2020}, Aboulker, Charbit and Naserasr showed that class of ($T$, \hyperlink{def:p4}{$\ora{P_4}$})-free oriented graphs have bounded dichromatic number. The authors also show that $\hyperlink{def:p4}{\ora{P_4}}$-free oriented graphs with clique number at most three have bounded dichromatic number.
    
    \item Let $\ora{K_t}$ denote the transitive tournament on $t$ vertices.
    In \cite{steiner2021}, Steiner showed that the class of ($\ora{K_3}$, \hyperlink{def:a4}{$\hyperlink{def:a4}{\ora{A_4}}$})-free oriented graphs has bounded dichromatic number.
  In the same paper Steiner asked whether the class of $(H,\ora{K_t})$-free oriented graphs has bounded dichromatic number for $t \geq 4$ and $H \in \{\hyperlink{def:p4}{\ora{P_4}}, \hyperlink{def:a4}{\ora{A_4}} \}$. We explain in the next subsection that our main result answers this question in the affirmative.
\end{itemize}

\subsection{Our contributions}
In this paper, we show that every orientation of $P_4$ is $\dichi$-bounding and thus the \hyperref[con:AlboukerGyarfasSumner4Digraphs2020]{ACN $\dichi$-boundedness conjecture} holds for all orientations of~$P_4$.
The \hyperref[con:AlboukerGyarfasSumner4Digraphs2020]{ACN $\dichi$-boundedness conjecture} is open for any orientation of $P_t$ for $t \geq 5$.
Our main novel result is that $\hyperlink{def:p4}{\ora{P_4}}$ and $\hyperlink{def:a4}{\ora{A_4}}$ are both $\dichi$-bounding.
Chudnovsky, Scott and Seymour showed that both \hyperlink{def:q4}{$\ora{Q_4}$} and \hyperlink{def:q4}{$\ora{Q_4}'$} are $\chi$-bounding and thus also $\dichi$-bounding in \cite{chudnovsky2019orientations}.
We include in a new proof that \hyperlink{def:q4}{$\ora{Q_4}$} and \hyperlink{def:q4}{$\ora{Q_4}'$} are both $\dichi$-bounding and improve the $\dichi$-binding function for the classes of \hyperlink{def:q4}{$\ora{Q_4}$}-free oriented graphs and \hyperlink{def:q4}{$\ora{Q_4}'$}-free oriented graphs.
To summarize, our main result is the following:

\begin{restatable}{theorem}{main}\label{thm:main}
Let $H$ be an oriented $P_4$.
Then, the class of $H$-free oriented graphs is $\dichi$-bounded.
In particular, for any $H$-free oriented graph $D$,
\[\dichi(D) \leq (\omega(D)+ 7)^{(\omega(D)+8.5)}.\]
\end{restatable}

Our result also answers the question of \cite{steiner2021} in the affirmative, that is, for $H \in \{\hyperlink{def:p4}{\ora{P_4}}, \hyperlink{def:a4}{\ora{A_4}} \}$ and any $k \geq 4$ the class of $H$-free oriented graphs not containing a transitive tournament of order $k$ has bounded dichromatic number.
Indeed, any tournament of order $2^{k-1}$ must contain a transitive tournament of order $k$. 
Thus, forbidding a given transitive tournament forbids any large enough tournament.
Note that there is no analogous result for non-transitive tournaments since all sub-tournaments of a transitive tournament are transitive.
The conjectures raised in \cite{AlboukerGyarfasSumner4Digraphs2020} are aimed at characterizing \emph{heroic} sets, that is, sets $\mathcal{F}$ such that digraphs (allowing digons) forbidding all elements of $\mathcal{F}$ have bounded dichromatic number.
If we ignore the degenerate cases where heroic sets include the empty graph or the graph consisting of a single vertex, there are no heroic sets of size one and the only heroic set of size two consists of an arc and a digon.
Therefore, heroic sets of size three are the smallest interesting case.
Hence, our work in this paper can be seen as a continuation of the investigation of \cite{AlboukerGyarfasSumner4Digraphs2020} into heroic sets of order three containing a digon.
In the language of heroic sets, if the class of all $H$-free oriented graphs is $\dichi$-bounded, then the set consisting of $H$, a digon, and a transitive tournament is a heroic set.
Then, Theorem~\ref{thm:main} can be restated by saying that the heroic sets of the form $\{\olra{K_2},H,K\}$, where $\olra{K_2}$ denotes a digon, $H$ is an orientation of $P_4$, are exactly those where $K$ is a transitive tournament.

\paragraph{Structure of the paper and proof overview.}
Let $H$ be any orientation of $P_4$.
We prove Theorem~\ref{thm:main} by induction on the clique number.
We fix an integer $\omega(D) \geq 2$.
We define a function $f$ and assume that $H$-free oriented graphs with clique number $\omega'$ where $1 \leq \omega' < \omega$ have dichromatic number at most $f(\omega')$.
We then consider an oriented graph with clique number $\omega$ and show that $D$ can be dicolored using at most $f(\omega)$ colors.

Our strategy to bound $\dichi(D)$ crucially relies on a tool we call \emph{dipolar sets} which were introduced by the name \say{nice sets} in \cite{AlboukerGyarfasSumner4Digraphs2020}.
Dipolar sets have the following useful property \cite{AlboukerGyarfasSumner4Digraphs2020}: In order to bound the dichromatic number of a class of oriented graphs closed under taking induced subgraphs, it suffices to exhibit a dipolar set of bounded dichromatic number for each of the members in the class.
We give a few preliminary observations as well as an introduction to dipolar sets in Section~\ref{sec:preliminaries}.

In Section~\ref{sec:nicesets}, we show how to construct a dipolar set for any $H$-free oriented graph $D$ of clique number $\omega$.
Our goal is to bound the dichromatic number of this set.
The backbone of our construction is an object we call a closed tournament.
\begin{definition}[\textcolor{bordeaux}{path-minimizing closed tournament}]
\label{def:closedtournament}
We say $K$ and $P$ \emph{form a closed tournament} $C = K \cup V(P)$
if $K$ is a tournament of maximum order and $P$ is a directed path from a source component to a sink component of the directed graph induced by $K$.

We say $K$ and $P$ form \emph{path-minimizing closed tournament} if $|P|$ is minimized amongst all choices of $K, P$ that form a closed tournament.
\end{definition}

It follows from the definition of closed tournament that the graph induced by a closed tournament is strongly connected and that every strongly connected oriented graph has a \hyperref[def:closedtournament]{path-minimizing} closed tournament.
We will define a set $S$ consisting of the closed neighborhood of a \hyperref[def:closedtournament]{path-minimizing} closed tournament $C$ and a subset of the second neighbors of $C$.
We will show that if $D$ is $H$-free, then $S$ is a dipolar set.
This proof will rely heavily on the fact that $C$ is strongly connected.

The strong connectivity of $C$ is a powerful property in showing that $S$ is a dipolar set.
However, ensuring $C$ is strongly connected by adding $P$ to $K$ makes it harder to bound the dichromatic number of $N(C)$.
We explain in Section~\ref{sec:preliminaries} that we can easily bound the dichromatic number of the first neighborhood of any bounded cardinality set.
Unfortunately, we have no control over the cardinality of $P$ in a path-minimum closed tournament.
In fact, $P$, and thus $C$, might be arbitrarily large with respect to $\omega$.
This significantly increases the difficulty of the task of bounding the dichromatic number of $N(C)$.
Fortunately, since $D$ is $H$-free and we may choose $C$ to be a \hyperref[def:closedtournament]{\emph{path-minimizing} closed tournament}, there are a lot of restrictions on what arcs may exist between vertices of $N(C)$. Ultimately able to exploit these restrictions to bound the dichromatic number of $N(C)$. 

Interestingly, we can define $S$ and prove that it is a dipolar set in the same way for each possible choice of an oriented $P_4$. We describe our construction of a dipolar set $S$ in Section~\ref{sec:nicesets}.
However, we used different (but similar) proofs to show that $S$ has bounded dichromatic number for $H = \hyperlink{def:p4}{\ora{P_4}}$, $\hyperlink{def:a4}{\ora{A_4}},$ and $\hyperlink{def:q4}{\ora{Q_4}}$. The proof that $S$ has bounded dichromatic number when $H$ is $\hyperlink{def:q4}{\ora{Q_4}}$ implies the result when $H$ is $\hyperlink{def:q4}{\ora{Q_4}'}$.

In Section~\ref{sec:common}, we bound the dichromatic number of $C$, the vertices of $S$ in the second neighborhood of $C$, $N(K)$ for $H$-free graphs where $H$ is an arbitrary choice of an orientated $P_4$.
In Section~\ref{sec:firstnbd}, we bound the dichromatic number of the vertices in $S$ not handled in Section~\ref{sec:common}.
These remaining vertices are the set $N(P) \setminus N[K]$.
Here we use separate (but similar) proofs for $H = \hyperlink{def:p4}{\ora{P_4}}, \hyperlink{def:a4}{\ora{A_4}}, \hyperlink{def:q4}{\ora{Q_4}}$.
In Section~\ref{sec:gettingthebound}, we put the pieces together to obtain our main result that any orientation of $P_4$ is $\dichi$-bounding.
We discuss some related open questions in Section~\ref{sec:conclusion}.

\section{Preliminaries}\label{sec:preliminaries}

In this section, we lay the groundwork for our proof by making a few observations useful in later sections and introducing dipolar sets.
In the rest of the paper, we will only consider strongly connected oriented graphs since the dichromatic number of an oriented graph is equal to the maximum dichromatic number of one of its strongly connected components. In particular, we will work with the following assumptions:

\begin{scenario}[Inductive Hypothesis]
\label{inductivehypothesis}
Let $H$ be an oriented $P_4$ and let $\omega > 1$ be an integer.
We let $\gamma$ be the maximum of $\dichi(D')$ over every $H$-free oriented graph $D'$ satisfying $\omega(D') < \omega$.
We assume $\gamma$ is finite.
We let $D$ be an $H$-free oriented graph with clique number $\omega$ and assume $D$ is strongly connected. 
\end{scenario}

We will aim to bound the $\dichi(D)$ in terms of $\gamma$ and $\omega$.
We begin with some easy observations about the dichromatic number of the neighborhood of any sets of vertices in $D$. 
For any vertex $v \in V(D)$, by definition $\omega(D[N(v)]) \leq \omega -1$ as otherwise $D$ would contain a tournament of size greater than $\omega$.
Hence, for any $v \in V(D)$, $\dichi(N(v)) \leq \gamma$.
This can be directly extended to bounding the dichromatic number of the neighborhood of a set of a given size as follows:

\begin{observation}\label{obs:neighbourhoodtrivi}
  Let $D$ be an oriented graph and let $\gamma$ be the maximum value of $\dichi(N(v))$ for any $v \in V(D)$.
  Then every $X\subseteq V(D)$ satisfies:
  \[
    \dichi(N(X)) \leq \dichi\left(\bigcup_{x \in X} N(x)\right) \leq |X|\cdot  \gamma.
\]
\end{observation}

We now formally define dipolar sets, one of the main tools used in this paper.
Note, dipolar sets were first introduced in \cite{AboulkerChordalDirectedGraphsAreNotChiBounded2022} as \say{nice sets}.

\begin{definition}[\textcolor{bordeaux}{dipolar set}]
\label{def:niceset}
A \emph{dipolar set} of an oriented graph $D$ is a nonempty subset $S \subseteq V(D)$ that can be partitioned into $S^+, S^-$ such that no vertex in $S^+$ has an out-neighbor in $V(D \setminus S)$ and no vertex in $S^-$ has an in-neighbor in $V(D \setminus S)$.
\end{definition}

We will use the following lemma from \cite{AboulkerChordalDirectedGraphsAreNotChiBounded2022} which reduces the problem of bounding the dichromatic number of $D$ to bounding the dichromatic number of a dipolar set in every induced oriented subgraph of $D$.
\begin{lemma}[Lemma 17 in \cite{AboulkerChordalDirectedGraphsAreNotChiBounded2022}]\label{lem:nice-set}
Let $\D$ be a family of oriented graphs closed under taking induced subgraphs.
Suppose there exists a constant $c$ such that every $D \in \D$ has a dipolar set $S$ with $\dichi(S) \leq c$.
Then every $D \in \D$ satisfies $\dichi(D) \leq 2c$. 
\end{lemma}
 
\section{Building a dipolar set}\label{sec:nicesets}
In this section we give a construction for a \hyperref[def:niceset]{dipolar} set in an $H$-free oriented graph $D$ where $H$ is a oriented $P_4$.
We will then show that the \hyperref[def:niceset]{dipolar} set we construct has bounded dichromatic number if $D$ satisfies the properties given in Scenario~\ref{inductivehypothesis}.

\subsection{Closed Tournaments}

The simplest case for our construction is when $D$ contains a \emph{strongly connected} tournament $J$ of order $\omega(D)$.
Then, we can build a \hyperref[def:niceset]{dipolar} set consisting of the union of $J$ and a subset of vertices at distance at most two from $K$.

Let $K$ be a tournament of order $\omega(D)$ contained in $D$.
By definition every vertex $v \in N(K)$ has a non-neighbor in $K$.
Hence, the graph underlying $D[K \cup \{v\}]$ contains an induced $P_3$.
Now, suppose $K$ is strongly connected.
Then we get an even more powerful property: 
Since $K$ is strongly connected there is both an arc from $K \setminus N(v)$ to $N(v) \cap K$ and to $K \setminus N(v)$ from $N(v) \cap K$.
This means that $D[K \cup \{v \}]$ contains an induced $P_3$ starting at $v$ whose last edge is oriented as $\rightarrow$ and an induced $P_3$ starting at $v$ whose last edge is oriented as $\leftarrow$.
This property will give us more power to build specific induced orientations of $P_3$ in $N[K]$.
In particular, this restricts the way vertices at distance at most two interact with the rest of the graph and allows us to choose a \hyperref[def:niceset]{dipolar} set.

To overcome the fact that $D$ may not contain a strongly connected tournament of order $\omega(D)$, we use closed tournaments.
By definition of closed tournament every strongly connected oriented graph has a path-minimizing closed tournament. 
We will base our construction of a \hyperref[def:niceset]{dipolar} set on some path-minimizing tournament in order to gain some additional structure that we can use to bound the dichromatic number of our \hyperref[def:niceset]{dipolar} set.
In the next subsection we formally give the definition of our \hyperref[def:niceset]{dipolar} set.

\subsection{Extending a closed tournament into a dipolar set}
In order to build a \hyperref[def:niceset]{dipolar} set from a closed tournament, we need to make some distinctions between different types of neighbors of a set of vertices.
\hypertarget{def:strongnbd}{For a set of vertices $A$ and $v \in N(A)$ we say $v$ is a \emph{strong neighbor} of $A$ if $v$ has both an in-neighbor and an out-neighbor in $A$. Then, the \textcolor{bordeaux}{\emph{strong neighborhood}} of $A$ is the set of strong neighbors of $A$.}

Given a closed tournament $C$, we let $X$ denote the set of strong neighbors of $C$.
The following lemma proves that $N[C \cup X]$ is a \hyperref[def:niceset]{dipolar} set.
\begin{lemma}\label{lem:buildinganiceset}
Let $H$ be an orientation of $P_4$ and $D$ be an $H$-free oriented graph.
Let $C$ be a closed tournament in $D$ and let $X$ denote the \hyperlink{def:strongnbd}{strong neighborhood} of $C$.  
Then $N[C \cup X]$ is a \hyperref[def:niceset]{dipolar} set.
\end{lemma}
\begin{figure}[t!]
    \centering
    \includegraphics[scale=0.3]{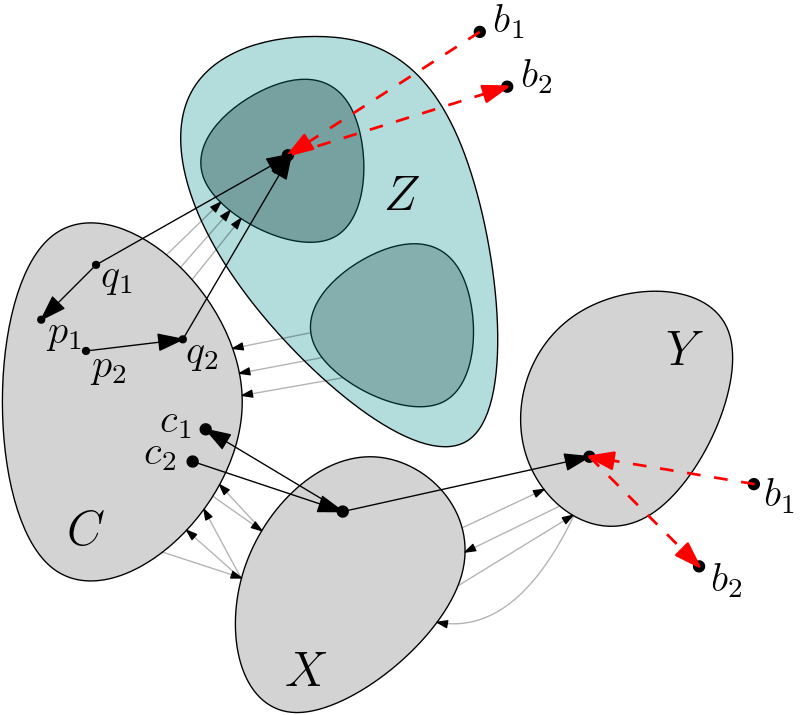}
    \caption{An illustration of the extension of a closed tournament $C$ into the \protect\hyperref[def:niceset]{dipolar} set $N[C \cup X]$.
    Highlighted in blue, $Z$ consists of neighbors of $C$ that are not strong, i.e., do not have both an in-neighbor and an out-neighbor in $C$.
    The set $X$ consists of the strong neighborhood of $C$, while set $Y$ contains all neighbors of $X$ not in $N[C]$. 
    Note that arcs between $Z$ and $X$ or $Y$ are not represented here.
    In Lemma~\ref{lem:buildinganiceset}, we prove that if there is some vertex in $N[C \cup X]$ with both an in-neighbor and an out-neighbor in the rest of the oriented graph (drawn in dashed red), then $N[C \cup X]$ contains all orientations of $P_4$ as an induced oriented subgraph.}
    \label{fig:neighborhoodClique}
\end{figure}

\begin{proof}
Let $Z$ denote the neighbors of $C$ that are not strong, and let $Y = N(X) \backslash N[C]$.
These sets satisfy $N[C \cup X] = C \cup X \cup Z \cup Y$ and the graph on $N[C \cup X]$ is illustrated in Figure~\ref{fig:neighborhoodClique}.

Then by definition, $N[C \cup X] = N[C] \cup Y$ and the only vertices of $N[C \cup X]$ with neighbors in $V(D) \setminus N[C \cup X]$ are in $Y \cup Z$.
Suppose for a contradiction that some $v \in Z \cup Y$ has both an in-neighbor $b_1$ and an out-neighbor $b_2$ in $V \setminus (N[C] \cup Y)$.
Let us first deal with the case where $v \in Y$.
\stmt{If $v \in Y$, then $D[N[C] \cup Y]$ contains $H$. \label{Ycase}}
Suppose $v \in Y$. Then, by definition, the following statements all hold:
\begin{itemize}
    \item There is some $x \in X$ such that $x$ and $v$ are adjacent.
    \item There are vertices $c_1, c_2 \in C$ where $c_1$ is an in-neighbor of $x$ and $c_2$ is an out-neighbor of $x$.
    \item $b_1, b_2$ are not adjacent to any of $x, c_1, c_2$.
\end{itemize}
Thus for some choice of $i,j \in \{1,2\}$ the set $\{c_i, x, v, b_j \}$ induces a copy of $H$. (See Figure~\ref{fig:neighborhoodClique}.)
This proves (\ref{Ycase}).
\\
\\
Since $D$ is an $H$-free oriented graph, it follows from (\ref{Ycase}) that $v \in Z$. Then by definition of $Z$, the neighbors of $v$ in $C$ are either all in-neighbors of $v$ or all out-neighbor of $v$.

\stmt{There exist arcs $(q_1, p_1), (p_2, q_2) \in E(C)$ such that $v$ is adjacent to $q_1, q_2$ and non-adjacent to $p_1, p_2$.
\label{edges}}
It follows from the fact that $\omega(C) = \omega(D)$ that $v$ has some non-neighbor in $C$.
Since $C$ is strongly connected, $N(v) \cap C$ must have both an incoming arc and an outgoing arc from $C \setminus N(v)$.
Let $p_1,p_2$ be vertices of $C \backslash N(v)$ witnessing this fact and let $q_1,q_2$ their respective neighbors in $N(v) \cap C$. This proves (\ref{edges}).
\\
\\
It follows that for some $i,j \in \{1,2\}$ the graph induced by $\{p_i, q_i, v, b_j\}$ is a copy of $H$, a contradiction. (See Figure~\ref{fig:neighborhoodClique}).
\end{proof}

\section{First steps towards bounding the dichromatic number of $N[C \cup X]$}
\label{sec:common}

As usual, we suppose $D$ satisfies the assumptions given in Scenario~\ref{inductivehypothesis} all hold.
We choose a tournament of order $\omega(D)$ and a directed path $P$ that form a \hyperref[def:closedtournament]{path-minimizing} tournament $C$ in $D$.
Let $X$ be the \hyperlink{def:strongnbd}{strong neighborhood} of $C$ and $Y = N(X) \setminus N[C]$ as before.
In the previous section we showed that $N[C \cup X]$ is a \hyperref[def:niceset]{dipolar} set.
Thus, by Lemma~\ref{lem:nice-set} we can prove that all orientations of $P_4$ are $\dichi$-bounding by proving that $\dichi(N[C \cup X])$ is bounded in terms of $\omega(D)$ and $\gamma$, the maximum value of $\dichi(D')$ for any $H$-free $D'$ with clique number less than $\omega(D)$.

By definition, $$\dichi(N[C \cup X]) \leq \dichi(N[K]) +  \dichi(V(P)) + \dichi(N(P) \setminus N[K]) + \dichi(Y).$$
We will bound $\dichi(N[C \cup X])$ by bounding each of the terms on the right-hand side of the equation.
We bound the dichromatic number of $N[K]$ and $P$ in Subsection~\ref{sub:PNK} and we bound the dichromatic number of $Y$ in Subsection~\ref{sub:second}. We are able to use the same techniques for each choice of $H$ when proving these bounds.

As already hinted, bounding $\dichi(N[C \cup X])$ is non-trivial because we have no control over the cardinality of $P$. 
Hence, we cannot obtain a useful bound on $\dichi(N(P) \setminus N[K])$ by simply applying Observation~\ref{obs:neighbourhoodtrivi}.
In the next section, we will show how to bound $\dichi(N(P) \setminus N[K])$. We will require separate proofs for $H = \hyperlink{def:q4}{\ora{Q_4}}, \hyperlink{def:p4}{\ora{P_4}}, \hyperlink{def:a4}{\ora{A_4}}$.

\subsection{Bounding the dichromatic number of $V(P)$ and $N[K]$}\label{sub:PNK}

We bound the dichromatic number of $V(P)$ and $N[K]$ by an easy observation about \say{forward-induced} paths.
\hypertarget{def:forwardinduced}{We say a directed path $p_1 \to p_2 \to \dots \to p_t$ is \textcolor{bordeaux}{\emph{forward-induced}} if no arc of the form $(p_i, p_j)$ exists where $j > i + 1$ and $i, j \in \{1,2, \dots, t \}$.}

\begin{observation}\label{obs:dichi-shortest-path}
Let $D$ be an oriented graph and let $P \subseteq D$ be a \hyperlink{def:forwardinduced}{forward-induced} directed path.
Then $\dichi(P) \leq 2$.
\end{observation}
\begin{proof}
Let the vertices of $P$ be $p_1 \to p_2 \to \ldots \to p_\ell$, in order.
We assign colors to the vertices of $P$ by alternating the colors along $P$.
Suppose there is some monochromatic directed cycle $Q$ in the oriented graph induced by $V(P)$.
Then $Q$ contains no arc of $P$.
Hence $Q$ must contain some arc $(p_i, p_j)$ with $i,j \in \{1,2, \dots, \ell\}$ and $j > i +1$, contradicting the defintion of forwards-induced.
\end{proof}

Now, we turn to bounding the dichromatic number of our dipolar set, $N[C \cup X]$ by bounding $\dichi(N[K] \cup V(P))$.
\begin{observation}\label{obs:dichi-closed-clique-nbd}
Let $D$ be an oriented graph satisfying $\dichi(N(v)) \leq \gamma$ for all $v \in V$.
Let $K$ be a maximum tournament and $P$ be a directed path in $D$ such that $K$ and $P$ form a  \hyperref[def:closedtournament]{path-minimizing closed tournament} $C$ in $D$.
Then, $\dichi(N[K]) \leq \omega \cdot \gamma$.
Moreover, $$\dichi(N[C]) \leq \dichi(N(P) \setminus N[K]) + \omega \cdot \gamma + 2.$$
\end{observation}
\begin{proof}
Since $|K| = \omega(D) > 1$, we have $N[K] = \bigcup_{x \in K} N(x)$, and hence $\dichi(N[K]) \leq \omega \cdot \gamma$
by Observation~\ref{obs:neighbourhoodtrivi}. This proves the first statement.
By definition, $N[C] = (N(P) \setminus N[K]) \cup N[K] \cup P$.
Since $C$ is \hyperref[def:closedtournament]{path-minimizing}, $P$ is \hyperlink{def:forwardinduced}{forward-induced}.
Thus, we obtain the second statement by Observation~\ref{obs:dichi-shortest-path}.
\end{proof}

Thus, it only remains to bound the dichromatic number of $Y$ and $N(P) \setminus N[K]$ in order to bound the dichromatic number of our \hyperref[def:niceset]{dipolar} set $N[C \cup X]$.

\subsection{Bounding the dichromatic number of $Y$}\label{sub:second}

In this subsection, we bound the dichromatic number of $Y = N(X) \backslash N[C]$. We first state a more general lemma, which gives the bound on $\dichi(Y)$ as a direct corollary.

\begin{lemma}\label{lem:QRS}
Let $H$ be an oriented $P_4$ and let $D$ be an $H$-free oriented graph.
Suppose there is a partition of $V(D)$ into sets $Q, R, S$ such that there is no arc between $Q$ and $S$, every $r \in R$ has both an in-neighbor and an out-neighbor in $Q$ and every $s \in S$ has a neighbor in $R$.
Let $\gamma$ be an integer such that for every $r\in R$, we have $\dichi(N(r)) \leq \gamma$.
Then $\dichi(S) \leq 2\gamma$.
\end{lemma}

\begin{figure}
    \centering
    \includegraphics[scale=0.2]{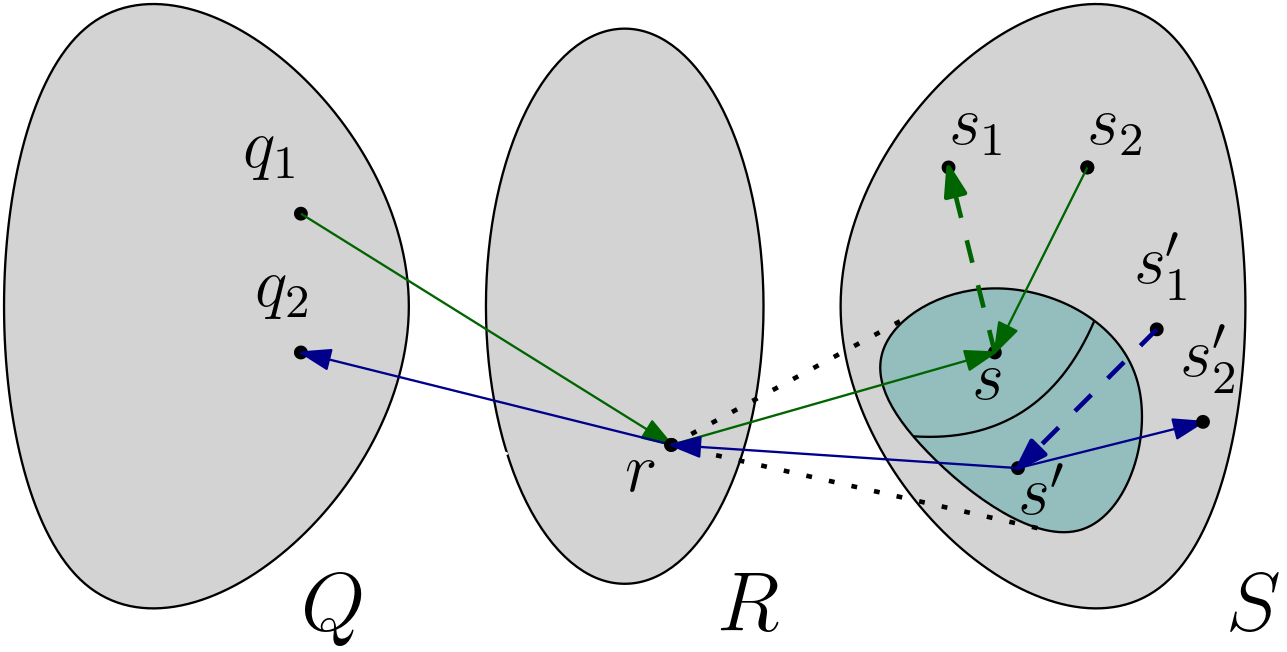}
    \caption{
    Vertex sets $Q,R,S$, such that no arc lies between $Q$ and $S$, the vertices in $R$ are all \protect\hyperlink{def:strongnbd}{strong neighbors} of $Q$ and $S$ is a subset of neighbors of $R$. We illustrate the case of graphs forbidding $\protect\hyperlink{def:p4}{\protect\ora{P_4}}$.
    Other orientations behave symmetrically. In dark green, a vertex $r \in R$ is depicted with an out-neighbor $s \in S$. Then if $s$ has an out-neighbor $s_1 \in S \setminus N(r)$ there would be an induced $\protect\hyperlink{def:p4}{\protect\ora{P_4}}$, a contradiction. Symmetrically in dark blue, an in-neighbor $s' \in S$ of $r$ cannot admit an in-neighbor $s_2' \in S \backslash N(r)$.
     }
    \label{fig:QRS}
\end{figure}

\begin{proof}
We proceed by induction on $|S|$.
Suppose $S \neq \emptyset$. 
Then there is some $r \in R$ that has a neighbor in $S$.
\stmt{We may partition $N(r) \cap S$ into two sets $S_1, S_2$ such that no vertex in $S_1$ has an in-neighbor in $S \setminus N(r)$ and no vertex in $S_2$ has an out-neighbor in $S \setminus N(r)$. \label{stmt:S}}
By definition $r$ has an in-neighbor $q_1$ and an out-neighbor $q_2$ in $Q$. 
Suppose some $s \in N(r)\cap S$ has both an in-neighbor $s_i$ and an out-neighbor $s_j$ in $S \setminus N(r)$.
Then there is an copy of $H$ induced by $\{q_i, r, s, s_j\}$ for some choice of $i,j \in \{1,2\}$, a contradiction. (See Figure~\ref{fig:QRS}.) This proves (\ref{stmt:S}).
\\
\\
Let $D_1, D_2$ be two disjoint sets of $\gamma$ colors each.
By induction we can dicolor $S \setminus N(r)$ with $D_1 \cup D_2$.
By (\ref{stmt:S}), we can extend this coloring to a dicoloring of $S$ by dicoloring $S_i$ with colors from $D_i$ for $i \in \{1,2\}$.
\end{proof}

\begin{corollary}\label{corr:Y-dichi}
Let $H$ be an oriented $P_4$ and let $D$ be a strongly connected $H$-free oriented graph.
Let $K$ be a maximum tournament and $P$ be a directed path in $D$ such that $K, P$ form a \hyperref[def:closedtournament]{path-minimizing closed tournament}.
Let $\gamma$ denote the maximum value of $\dichi(N(v))$ for any $v \in D$.
Then, $\dichi(N(K)) \leq \omega(D) \cdot \gamma$.
Let $X$ denote the \hyperlink{def:strongnbd}{strong neighborhood} $C$ and let $Y$ denote the set $N(X) \setminus N[C]$.
Then, $$\dichi(Y) \leq 2\gamma.$$ 
Moreover,
$$\dichi(N[C \cup X]) \leq \dichi(N(P) \setminus N[K]) + (\omega(D) + 2)\cdot \gamma  + 2.$$
\end{corollary}
\begin{proof}
By definition we may applying Lemma~\ref{lem:QRS} to the induced subgraph $D[C \cup X \cup Y]$ with $Q \coloneqq C$, $R \coloneqq X$ and $S \coloneqq Y$ (see Figure~\ref{fig:neighborhoodClique}). 
Hence, $\dichi(Y) \leq 2\gamma$.
By definition, $N[C \cup X] = Y \cup N[C]$.
Hence, $$\dichi(N[C \cup X]) \leq \dichi(Y) + \dichi(N[C]) \leq 2 \gamma + \dichi(N[C]).$$
Then by Observation~\ref{obs:dichi-closed-clique-nbd} and since we obtain:
\[\dichi(N[C \cup X]) \leq \dichi(N(P) \setminus N[K]) + (\omega(D) + 2) \cdot \gamma + 2. \qedhere\]
\end{proof}

Thus, if we can bound $\dichi(N(P) \setminus N[K])$ we can bound the dichromatic number of our \hyperref[def:niceset]{dipolar} set $N[C \cup X]$ by Corollary~\ref{corr:Y-dichi}.
We will handle this in the next section.

\section{Completing the bound on the dichromatic number of our dipolar set}
\label{sec:firstnbd}
In this section, we will prove a bound on the dichromatic number of $N(P) \setminus N[K]$ where $K$ is a maximum tournament and $P$ is a directed path that forms a \hyperref[def:closedtournament]{path-minimizing closed tournament} $C$ in an oriented graph that forbids some orientation of $P_4$.
By Corollary~\ref{corr:Y-dichi}, this will imply that every oriented graph has which forbids some orientation of $P_4$ has a \hyperref[def:niceset]{dipolar} set of bounded dichromatic number.
Thus, by Lemma~\ref{lem:nice-set}, this will give us our main result.

By definition of \pathminimizingclosedtournament, $P$ is a \forwardinduced directed path.
In Subsection~\ref{sub:firstStructure}, we start by giving some structural properties on properties of the neighborhood of forwards-induced paths.
Then, in Subsection~\ref{sub:firstQ4} we show how to use these properties to bound the dichromatic number of the first neighborhood of $C$ for $\hyperlink{def:q4}{\ora{Q_4}}$-free graphs. (Recall, the bound for $\hyperlink{def:q4}{\ora{Q_4}}$-free oriented graphs implies the bound for $\hyperlink{def:q4}{\ora{Q_4}'}$-free oriented graphs.)

When $H$ is one of the other two orientations, $\hyperlink{def:p4}{\ora{P_4}}$ and $\hyperlink{def:a4}{\ora{A_4}}$, we required a finer analysis of $N(P)$ in order to bound $\dichi(N(P))$.
We handle this case in Subsections~\ref{sub:firstWpWa}--\ref{sub:firstRa}.

\subsection{Forbidden arcs among neighbors of a forward-induced directed path}\label{sub:firstStructure}

We define two partitions of the first neighborhood of a directed path and show how to forbid some of the arcs between classes of each partition in an $H$-free oriented graph.

\begin{definition}
\label{def:FL}
Let $P = p_1 \to p_2 \to \dots \to p_ \ell$ be a \hyperlink{def:forwardinduced}{forward-induced} directed path in an oriented graph.
For brevity, for any $v \in N(P)$ and $i,j \in \{1,2, \dots, \ell\}$
we say $p_i$ is the \emph{first} neighbor of $v$ on $P$ if $v$ is adjacent to $p_i$ and non-adjacent to $p_{i'}$ for each $1 \leq i' < i \leq \ell$.
Similarly, $p_j$ is the \emph{last} neighbor of $v$ on $P$ if $v$ is adjacent to $p_j$ and non-adjacent to each $p_{j'}$ for each $1 \leq j < j' \leq \ell$.
We will define two partitions of $N(P)$ according to their first and last neighbors in $V(P)$, respectively.
\begin{itemize}
\item For each $i \in \{1,2, \dots, \ell\}$ we say $v \in N(P)$ is in $F_i$ if $p_i$ is the first neighbor of $v$ on $P$.
This yields partition $(F_1, F_2, \dots, F_\ell)$, which we call \emph{\color{bordeaux} the partition of $N(P)$ by \emph{first attachment} (on $P$)}.
\item Symmetrically, for each $j \in \{1,2, \dots, \ell\}$ we say $v \in L_j$ if $p_j$ is the last neighbor of $v$ on $P$.
This yields partition $(L_1, L_2, \dots, L_\ell)$, which we call the \textcolor{bordeaux}{\emph{partition of $N(P)$ by last attachment (on $P$)}}.
\end{itemize}
For each $i \in \{1,2, \dots, \ell\}$ we refine each partition by dividing each $F_i, L_i$ into the in-neighbors and out-neighbors of $v$.
We define \textcolor{bordeaux}{$F_i^+$} and \textcolor{bordeaux}{$L_i^+$} to be the sets consisting of all the in-neighbors of $p_i$ in $F_i, L_i$, respectively.
Similarly, we define \textcolor{bordeaux}{$F_i^-$} and \textcolor{bordeaux}{$L_i^-$} to be the sets consisting of all the out-neighbors of $p_i$ in $F_i, L_i$, respectively.
\end{definition}

\begin{observation}\label{obs:dagNeighborhood}
Let $P = p_1 \to p_2 \to \dots \to p_ \ell$ be a forward-induced directed path in an oriented graph $D$.
Let $(F_1, F_2, \dots, F_\ell)$,  $(L_1, L_2, \dots, L_\ell)$ be the \hyperref[def:FL]{partitions of $N(P)$ by first attachment and last attachment} on $P$, respectively.
Let $2 \leq i < j \leq \ell - 1$.
Then the following statements all hold:
\begin{itemize}
    \item If $D$ is $\hyperlink{def:q4}{\ora{Q_4}}$-free, there are no arcs from $F_j$ to $F_i$.
    \item If $D$ is $\hyperlink{def:p4}{\ora{P_4}}$-free, there are no arcs from $F_i^-$ to $F_j$, and no arcs from $L_i$ to $L_j^+$.
    \item If $D$ is $\hyperlink{def:a4}{\ora{A_4}}$-free, there are no arcs from $F_i^+$ to $F_j$, and no arcs from $L_i$ to $L_j^-$.
\end{itemize}
\end{observation}
\begin{proof}
Let $2 \leq i < j \leq \ell - 1$. We prove each statement individually.
\stmt{If $D$ is $\hyperlink{def:q4}{\ora{Q_4}}$-free, there are no arcs from $F_j$ to $F_i$. \label{dag:q4}}
Suppose for some $v \in F_j$ and $w \in F_i$ that $(v, w) \in E(D)$.
Then the vertices $p_{i-1},p_i,w, v$, induce a $P_4$ in $D$ with orientation $p_{i-1} \rightarrow p_i \rightarrow w \leftarrow v$ or orientation $p_{i-1} \rightarrow p_i \leftarrow w \leftarrow v$ depending on whether $w \in F_i^+$ or $w \in F_i^-$.
In either case we obtain an induced $\hyperlink{def:q4}{\ora{Q_4}}$ on $p_{i-1},p_i,w, v$.
This proves (\ref{dag:q4}).
\stmt{If $D$ is $\hyperlink{def:p4}{\ora{P_4}}$-free, there are no arcs from $F_i^-$ to $F_j$, and no arcs from $L_i$ to $L_j^+$.\label{dag:p4}}
Suppose for some $v \in F_i^-$ and $w \in F_j$ that $(v, w) \in E(D)$.
Then $p_{i-1} \to p_i \to v \to w$ is an induced $\hyperlink{def:p4}{\ora{P_4}}$ (see the dark blue arcs in Figure~\ref{fig:neighborhoodpartition}). Hence $D$ is not $\hyperlink{def:p4}{\ora{P_4}}$-free. This proves the first part of the statement (\ref{dag:p4}).
The argument that there are no arcs from $L_i$ to $L_j^+$ in an $\hyperlink{def:p4}{\ora{P_4}}$-free graph is symmetric.
This proves (\ref{dag:p4}).
\begin{figure}[b!]
    \centering
    \includegraphics[scale=0.25]{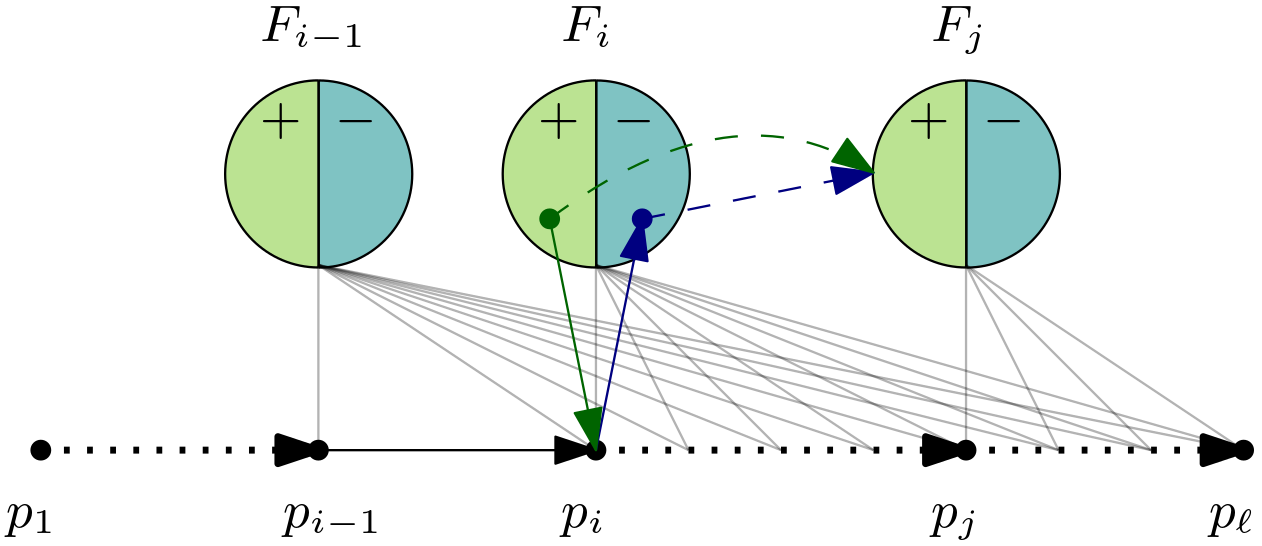}
    \caption{A shortest path directed path $P=p_1 \to ... \to p_{\ell}$ along with the partition $(F_1,...,F_{\ell})$ of $N(P)$ by first attachment on $P$. Note that the setting is symmetric for the partition of $N(P)$ by the last attachment. Each class of the partition $F_i$ is represented as a circle and further split into $F_i^+$ in green and $F_i^-$ in blue, all possible arcs towards $P$ are drawn in gray. An arc from $F_i^-$ to $F_j$ with $j > i$ would induce a $\protect\hyperlink{def:p4}{\protect\ora{P_4}}$ using $(p_{i-1},p_i)$, as highlighted in dark blue. An arc from $F_i^+$ to $F_j$ would induce a $\protect\hyperlink{def:a4}{\protect\ora{A_4}}$, represented in dark green.}
    \label{fig:neighborhoodpartition}
\end{figure}
\stmt{If $D$ is $\hyperlink{def:a4}{\ora{A_4}}$-free, there are no arcs from $F_i^+$ to $F_j$, and no arcs from $L_i$ to $L_j^-$. \label{dag:a4}}
By symmetry it is enough to show that if $D$ is $\hyperlink{def:a4}{\ora{A_4}}$-free then there is no arc from $F_i^+$ to $F_j$. 
Suppose for some $v \in F_i^+$ and $w \in F_j$ that $(v, w) \in E(D)$.
Then $p_{i-1} \to p_i \leftarrow v \rightarrow w$ is an induced $\hyperlink{def:a4}{\ora{A_4}}$ in $D$ (see the dark green arcs in Figure~\ref{fig:neighborhoodpartition}). This proves (\ref{dag:a4}).
\end{proof}

In the Subsection~\ref{sub:firstQ4} we use Observation~\ref{obs:dagNeighborhood} to bound the dichromatic number of $N(P) \setminus N[K]$ in the $\hyperlink{def:q4}{\ora{Q_4}}$-free case.
In the $\hyperlink{def:p4}{\ora{P_4}}$-free case and the $\hyperlink{def:a4}{\ora{A_4}}$-free case we need to perform a more careful analysis of $N(P) \setminus N[K]$ in order to bound its dichromatic number because the conditions guaranteed by Observation~\ref{obs:dagNeighborhood} are weaker in these two cases.
In Subsection~\ref{sub:firstWpWa}, we use Observation~\ref{obs:dagNeighborhood} to bound the dichromatic number of the following subsets of $N(P) \setminus N[K]$
\begin{equation} \label{def:wp}
        \textcolor{bordeaux}{W^p} = \left(F_2^- \cup F_3^- \cup \dots \cup F_{\ell-1}^- \right) \cup \left( L_2^+ \cup L_3^+ \cup \dots \cup L_{\ell-1}^+ \right) \\
    \end{equation}
when $D$ is $\hyperlink{def:p4}{\ora{P_4}}$-free and 
    \begin{equation}\label{def:wa}
        \textcolor{bordeaux}{W^a} = \left( F_2^+ \cup F_3^+ \cup \dots \cup F_{\ell-1}^+ \right) \cup \left( L_2^- \cup L_3^- \cup \dots \cup L_{\ell-1}^- \right) \\
    \end{equation}
when $D$ is $\hyperlink{def:a4}{\ora{A_4}}$-free. The vertices in $N(P) \setminus (N[K] \cup W^p)$ and $N(P) \setminus (N[K] \cup W^a)$ have restrictions on how they may have neighbors in $V(P)$.
We will use this to bound their dichromatic number in Subsections~\ref{sub:firstRp} and \ref{sub:firstRa}, respectively.

\subsection{The $\protect\hyperlink{def:q4}{\protect\ora{Q_4}}$-free case}\label{sub:firstQ4}

In section, we bound the dichromatic number of
a \pathminimizingclosedtournament  in $D$ when $D$ is a $\hyperlink{def:q4}{\ora{Q_4}}$-free oriented graph satisfying the conditions of Scenario~\ref{inductivehypothesis}.

\begin{lemma}\label{lem:Q4:firstNeighborsDichi}
Let $P$ be a \hyperlink{def:forwardinduced}{forward-induced} directed path in $D$ and $\gamma$ be an integer satisfying $\dichi(N(v))\leq \gamma$ for each $v \in V(P)$. 
Let the vertices of $P$ be $p_1 \to p_2 \to \dots \to p_\ell$, in order.
Then, if $D$ is $\hyperlink{def:q4}{\ora{Q_4}}$-free, $\dichi(N(P) \setminus N(\{p_1,p_{\ell}\})) \leq \gamma$.
\end{lemma}
\begin{proof}
Assume $D$ is $\hyperlink{def:q4}{\ora{Q_4}}$-free oriented graph.
Let $(F_1, F_2, \dots, F_\ell)$ be the \hyperref[def:FL]{partition of $N(P)$ by first attachment} on $P$.
By definition $N(P) \setminus N(\{p_1,p_{\ell}\}) \subseteq F_2 \cup F_3 \cup \dots \cup F_{\ell-1}$.
By Observation~\ref{obs:dagNeighborhood} every directed cycle in $D[F_2 \cup F_3 \cup \dots \cup F_{\ell-1}]$ is completely contained in $D[F_i]$ for some $i \in [2,\ell-1]$.
By definition $F_i \subseteq N(p_i)$ so $\dichi(F_i) \leq \gamma$ for each $i \in \{1, 2, \dots, \ell \}$.
Hence we may use the same set of $\gamma$ colors for each of $F_2, F_3, \dots , F_{\ell-1}$.
Thus, $\dichi(N(P) \setminus N(\{p_1, p_\ell\}) \leq \gamma$.
\end{proof}

Lemma~\ref{lem:Q4:firstNeighborsDichi} allows us to demonstrate a bound on our \hyperref[def:niceset]{dipolar} set $N[C \cup X]$ as follows:

\begin{lemma}
\label{lem:q4:nicesetbound}
Let $D$ be a strongly connected $\hyperlink{def:q4}{\ora{Q_4}}$-free oriented graph.
Let $\gamma$ be an integer satisfying $\dichi(N(v))\leq \gamma$ for each $v \in V(P)$.
Then $D$ has a \hyperref[def:niceset]{dipolar} set with dichromatic number at most $(\omega(D) + 3)\cdot \gamma + 2$
\end{lemma}
\begin{proof}
Let $K$ be a maximum tournament and $P$ be a directed path in $D$ such that $K$ and $P$ form a \hyperref[def:closedtournament]{path-minimizing closed tournament} $C$ in $D$.
Let $X$ denote the \hyperlink{def:strongnbd}{strong neighborhood} of $C$.
By Lemma~\ref{lem:buildinganiceset}, $N[C \cup X]$ is a \hyperref[def:niceset]{dipolar} set.
By Corollary~\ref{corr:Y-dichi} we obtain:
$$\dichi(N[C \cup X]) \leq \dichi(N(P) \setminus N[K]) + (\omega(D) + 2)\cdot \gamma + 2.$$
Let $p_1, p_\ell$ denote the ends of $P$.
Then by definition $p_1, p_\ell \in K$.
Hence the result follows by Lemma~\ref{lem:Q4:firstNeighborsDichi}.
\end{proof}

\subsection{The $\protect\ora{P_4}$-free case and the $\protect\ora{A_4}$-free case}
In this subsection, we bound the dichromatic number our \hyperref[def:niceset]{dipolar} set in the case where $D$ is $\hyperlink{def:p4}{\ora{P_4}}$-free or $\hyperlink{def:a4}{\ora{A_4}}$-free.

\subsubsection{Bounding $\protect\dichi(W^p)$ and $\protect\dichi(W^a)$}
\label{sub:firstWpWa}
In this subsection, we bound the dichromatic number of of \hyperref[def:wp]{$W^p$} and \hyperref[def:wa]{$W^a$} using Observation~\ref{obs:dagNeighborhood} in $\hyperlink{def:p4}{\ora{P_4}}$-free and $\hyperlink{def:a4}{\ora{A_4}}$-free oriented graphs, respectively. 

\begin{lemma}\label{lem:Wdichi}
Let $P$ be a \forwardinduced directed path in $D$ and $\gamma$ be an integer satisfying $\dichi(N(v))\leq \gamma$ for each $v \in V(P)$.
Let \hyperref[def:wp]{$W^p$} and \hyperref[def:wa]{$W^a$} be defined with respect to $P$.
Then, the following statements both hold:
\begin{itemize}
    \item If $D$ is $\hyperlink{def:p4}{\ora{P_4}}$-free, then $\dichi(W^p) \leq 2\gamma$.
    \item If $D$ is $\hyperlink{def:a4}{\ora{A_4}}$-free, then $\dichi(W^a) \leq 2\gamma$.
\end{itemize}
\end{lemma}
\begin{proof}
Let the vertices of $P$ be $p_1 \to p_2 \to \ldots \to p_\ell$, in order.
Let $(F_1, F_2, \dots, F_\ell)$ and $(L_1, L_2,$ $\dots, L_\ell)$ be the \hyperref[def:FL]{partitions of $N(P)$ by first attachment and last attachment} on $P$, respectively., respectively.

We begin by proving the first bullet.
Suppose $D$ is $\hyperlink{def:p4}{\ora{P_4}}$-free.
Then by Observation~\ref{obs:dagNeighborhood}, every directed cycle in $D[F_2^- \cup F_3^- \cup \dots \cup F_{\ell-1}^-]$ is completely contained in $D[F_i^-]$ for some $i \in [2, \ell-1]$. 
By assumption, $\dichi(N(p_i)) \leq \gamma$ for every $p_i \in P$.
Hence, we may use the same set of $\gamma$ colors for each of $F_2^-, F_3^-, \dots , F_{\ell-1}^-$.
So $\dichi(F_2^- \cup F_3^- \cup \dots \cup F_{\ell-1}^-) \leq \gamma$.
By symmetry, $\dichi(L_2^+ \cup L_2^+ \cup \dots \cup L_{\ell-1}^+) \leq \gamma$.
Therefore, since $\hyperref[def:wp]{W^p}$ is the union of these two sets, we obtain $\dichi(\hyperref[def:wp]{W^p}) \leq 2\gamma$.

The case is symmetric when $D$ is $\hyperlink{def:a4}{\ora{A_4}}$-free.
The third item of Observation~\ref{obs:dagNeighborhood} allows us to use the same set of colors for each of $F_2^+, F_3^+ , \dots , F_{\ell-1}^+$, and the same set of colors for each  of $L_2^- , L_3^-, \dots,  L_{\ell-1}^-$.
Hence, $\dichi(F_2^+ \cup F_3^+ \cup \dots \cup F_{\ell-1}^+) \leq \gamma$ and $\dichi(L_2^- \cup L_3^- \cup \dots \cup  L_{\ell-1}^-) \leq \gamma$.
Since $\hyperref[def:wa]{W^a}$ is the union of these two sets, $\dichi(\hyperref[def:wa]{W^a}) \leq 2\gamma$.
\end{proof}

\subsubsection{Completing the bound on the dichromatic number of our dipolar set in the $\protect\ora{P_4}$-free case}\label{sub:firstRp}
In this section, we will consider the dichromatic number of the following set of vertices.
\begin{definition} %
\label{def:rp}
Let $P$ be a \hyperlink{def:forwardinduced}{forward-induced} directed path in an oriented graph.
Let the vertices of $P$ be $p_1 \to p_2 \to \dots \to p_\ell$, in order.
We let $$\textcolor{bordeaux}{R^p} = N(P) \setminus (N(\{p_1,p_2,p_\ell\}) \cup W^p ).$$
where \hyperref[def:wp]{$W^p$} is defined with respect to $P$.
\end{definition}

We will assume that $D$ satisfies the conditions of Scenario~\ref{inductivehypothesis} with $H = \ora{P_4}$ for the remainder of Subsubsection~\ref{sub:firstRp}. In other words, $D$ is a strongly connected $\ora{P_4}$-free oriented graph with clique number $\omega$, and there is some finite $\gamma$ such that every $\hyperlink{def:p4}{\ora{P_4}}$-free oriented graph with clique number less than $\omega$ has dichromatic number at most $\gamma$.
Let $K$ be a maximum tournament and $P$ be a directed path in $D$ such that $K$ and $P$ form a  \hyperref[def:closedtournament]{path-minimizing closed tournament} $C$ in $D$.
Then, $N(P) \setminus N[K] \subseteq \hyperref[def:wp]{W^p} \cup \hyperref[def:rp]{R^p} \cup N(p_2)$ and we can bound $\dichi(\hyperref[def:wp]{W^p})$ in terms of $\omega(D)$ and $\gamma$.
Since $\omega(D) = \omega$, it follows that $\omega(N(v)) < \omega$ for each $v \in V(D)$. Hence, $\dichi(N(p_2)) \leq \gamma$.
Thus, by Lemma~\ref{lem:buildinganiceset} and Corollary~\ref{corr:Y-dichi}, we only need to bound $\dichi(\hyperref[def:rp]{R^p})$ in terms of $\omega$ and $\gamma$ in order to demonstrate that $D$ as a \hyperref[def:niceset]{dipolar} set of bounded dichromatic number.

By definition, \hyperref[def:wp]{$W^p$} is the set of vertices in $N(P) \setminus N(\{p_1, p_\ell\})$ whose first neighbor on $P$ is an in-neighbor or whose last neighbor in $V(P)$ is an out-neighbor.
Hence, $R^p$ consists exactly of the vertices in $N(P) \setminus N(\{p_1,p_2,p_{\ell}\})$ whose first neighbor in $V(P)$ is an out-neighbor and whose last neighbor in $V(P)$ is an in-neighbor..

We will show that since $C=K \cup V(P)$ is a  \hyperref[def:closedtournament]{\emph{path-minimizing} closed tournament} there is no tournament of order $\omega$ in $R^p$ and thus $\dichi(R^p) \leq \gamma$.
In particular, we will show that for a contradiction, if $R^p$ has a tournament $J$ of order $\omega$, then we can find a directed path $P'$ that is \emph{shorter than $P$} such that $J$ and $P'$ form a closed tournament.
In order to prove this, we will need the following lemma, which will allow us to exhibit a relatively short path between two adjacent vertices in $R^p$.
\begin{lemma}\label{lem:dp:paths}
Let $D$ be a $P_4$-free oriented graph.
Let $P=p_1 \to p_2 \to \ldots \to p_\ell$ be a \forwardinduced directed path in $D$.
Let \hyperref[def:rp]{$R^p$} be defined with respect to $P$.
 Let $v, w \in R^p$, if $(w,v) \in E(D)$, there is a directed path from $v$ to $w$ on at most $\max\{6, \ell-1\}$ vertices.
\end{lemma}

\begin{figure}[t!]
    \centering
    \includegraphics[scale=0.17]{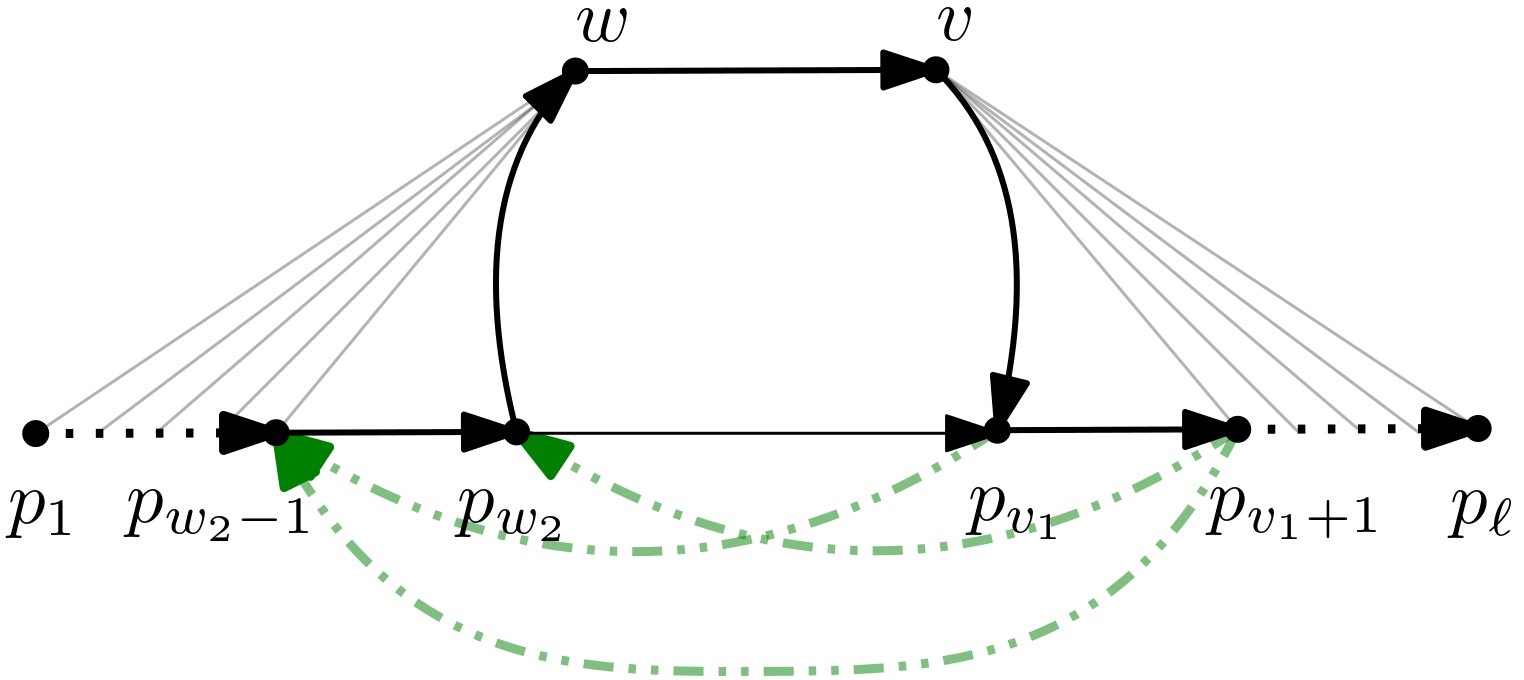}
    \caption{On the bottom, a shortest path $P$ in $D$, and an arc $(w,v)$ between neighbors of $P$ in $R^p$. Illustrated here is the case where the last neighbor $p_{w_2}$ of $w$ on $P$ appears just before the first neighbor $p_{v_1}$ of $v$, all possible arcs are shown in gray. Then, path $p_{w_2-1},p_{w_2},p_{v_1},p_{v_1 + 1}$ cannot induce a $\protect\hyperlink{def:p4}{\protect\ora{P_4}}$. Any arc possibly preventing this, shown in dash-dotted green, yields a path from $v$ to $w$ of length at most five.}
    \label{fig:shortcutP4}
\end{figure}

\begin{proof}
Let $p_{v_1}$ denote the first neighbor of $v$ in $V(P)$ and let $p_{w_2}$ denote the last neighbor of $w$ in $V(P)$.
Then, since $v,w \not \in \cup_{i=1}^\ell \hyperref[def:FL]{F_i^-} \cup \hyperref[def:FL]{L_i^+}$, the corresponding arcs are $(v, p_{v_1}),(p_{w_2}, w) \in E(D)$.
By definition of $R^p$, we obtain that $3 \leq v_1, w_2 \leq \ell -1$.
Hence, we may assume that $w_2 < v_1$, for otherwise $v \to p_{v_1} \to p_{v_1+1} \to \dots \to p_{w_2} \to w$ is a directed path from $v$ to $w$ with at most $\ell -1$ vertices, as desired.

Since $w_2 < v_1$ the vertices $v,w$ have no common neighbors in $V(P)$. 
Now, consider the directed path $p_{w_2} \to w \to v \to p_{v_1}$. Since $D$ is $\hyperlink{def:p4}{\ora{P_4}}$-free it cannot be induced.
Thus, $(p_{w_2}, p_{v_1}) \in E(D)$ or $(p_{v_1}, p_{w_2}) \in E(D)$.

Suppose that $(p_{v_1}, p_{w_2}) \in E(D)$. 
Then, $v \to p_{v_1} \to p_{w_2} \to w$ is a directed path from $v$ to $w$ of length three, as desired.
Hence we may assume that $(p_{w_2}, p_{v_1}) \in E(D)$.

Since $P$ is a \forwardinduced directed path  and $w_2 < v_1$, it follows that $v_1 = w_2 +1$.
Consider the directed path $p_{w_2-1} \to p_{w_2} \to p_{v_1} \to p_{v_1 + 1}$. Since $D$ is $\hyperlink{def:p4}{\ora{P_4}}$-free it cannot be induced.
Therefore, the vertices $p_{w_2-1}$ and $p_{v_1}$, the vertices $p_{w_2}$ and $p_{v_1+1}$, or the vertices $p_{w_2-1}$ and $p_{v_1+1}$ are adjacent.
Furthermore, since $P$ is a shortest path, this means that at least one of $(p_{v_1+1}, p_{w_2-1}), (p_{v_1+1}, p_{w_2}), (p_{v_1}, p_{w_2-1})$ is an arc of $D$, see Figure~\ref{fig:shortcutP4}. We consider each case separately:
\begin{itemize}
\item Suppose $(p_{v_1}, p_{w_2-1})$ is an arc of $D$. Then $v \to p_{v_1} \to p_{w_2-1} \to p_{w_2} \to w$ is a path of $D$.
\item Suppose $(p_{v_1+1}, p_{w_2})$ is an arc of $D$. Then $v \to p_{v_1} \to p_{v_1 + 1} \to p_{w_2} \to w$ is a path of $D$.
\item Suppose $(p_{v_1+1}, p_{w_2 -1})$ is an arc of $D$. Then $v \to p_{v_1} \to p_{v_1 + 1} \to p_{w_2-1} \to p_{w_2} \to w$ is a path of $D$.
\end{itemize}
In every case, the oriented graph induced by $\{v, p_{v_1} , p_{v_1 + 1}, p_{w_2-1}, p_{w_2}, w \}$ contains a directed path from $v$ to $w$ on at most six vertices.
Since one of the cases must hold, this completes the proof.
\end{proof}

With the last lemma in hand, we are ready to bound the dichromatic number of $R^p$.

\begin{lemma}\label{lem:RP-dichi}
Suppose $D$ satisfies the conditions of Scenario~\ref{inductivehypothesis} with $H = \ora{P_4}$.
Then $D$ contains a \hyperref[def:niceset]{dipolar} set of dichromatic number at most $(\omega + 6) \cdot \gamma + 2$.
\end{lemma}
\begin{proof}
Let $K$ be a maximum tournament and $P$ be a directed path in $D$ such that $K$ and $P$ \hyperref[def:closedtournament]{form a path-minimizing closed tournament} $C$ in $D$.
Then by Lemma~\ref{lem:buildinganiceset}, $N[C \cup X]$ is a \hyperref[def:niceset]{dipolar} set.
We will use Lemma~\ref{lem:dp:paths} to bound $\dichi(R^p)$.
Then we will combine this bound with the results from the previous sections to bound $\dichi(N[C \cup X])$.

\stmt{If $P$ contains at least 7 vertices, then $\dichi(R^p) \leq \gamma$. \label{rp:bound}}
Suppose $\ell \geq 7$.
If $\omega(D[R^p]) < \omega(D)$ then by assumption $\dichi(R^p) \leq \gamma$, so we may assume that there is an $\omega(D)$-tournament $J \subseteq R$.
Since $C$ is a minimum closed tournament and $P$ is non-empty, $J$ is not strongly connected.
Hence there must be exactly one strongly connected component of $J$ that is a sink and exactly one strongly connected component of $J$ that is a source (and they are not equal). 
Let $v$ be a vertex in the sink component of $J$ and $w$ be a vertex in the source component of $J$. Therefore, $(w,v) \in E(D)$.
Thus by Lemma~\ref{lem:dp:paths} there is a path $Q$ from $v$ to $w$ of length less than that that of $P$.
Hence, $J, P'$ form a closed tournament.
By definition since $K, P$ were chosen to form a \hyperref[def:closedtournament]{path-minimizing closed tournament} $P'$ cannot be shorter than $P$, a contradiction.
This proves (\ref{rp:bound}).
\stmt{$\dichi(N(P) \setminus N[K]) \leq 4 \gamma$ \label{dp:pathbound}}
Let the vertices of $P$ be $p_1 \to p_2 \to \ldots \to p_\ell$, in order.
By definition, $p_1, p_\ell \in K$.
If $\ell \leq 6$, then by Observation~\ref{obs:neighbourhoodtrivi}, $N(P) \setminus N[K] \leq 4\gamma$, as desired. Hence, we may assume this is not the case.

Let \hyperref[def:wp]{$W^p$}, \hyperref[def:rp]{$R^p$} be defined with respect to $P$.
Then, $$\dichi(N(P) \setminus N[K]) \leq \dichi(W^p) + \dichi (R^p) + \dichi(N(p_2)).$$
Hence, by Lemma~\ref{lem:Wdichi}, Observation~\ref{obs:neighbourhoodtrivi} and (\ref{rp:bound}) we obtain
$\dichi(N(P) \setminus N[K] \leq 4 \gamma$.
This proves (\ref{dp:pathbound}).
\\
\\
By combining (\ref{dp:pathbound}) with Corollary~\ref{corr:Y-dichi}, we have \[\dichi(N[C \cup X]) \leq \dichi(N(P) \setminus N[K]) + (\omega + 2) \cdot \gamma + 2 \leq (\omega + 6) \cdot \gamma + 2. \qedhere\]
\end{proof}

\subsubsection{Completing the bound on $\protect\dichi$ of our {dipolar} set in the $\protect\ora{A_4}$-free case}\label{sub:firstRa}
In this section, we prove a bound on the remaining vertices of $N(P) \setminus N[K]$ and use the results of the previous sections to show that $D$ contains a \hyperref[def:niceset]{dipolar} set of bounded dichromatic number in the $\hyperlink{def:a4}{\ora{A_4}}$-free case.
\begin{definition}\label{def:ra}
Let $P$ be a shortest path with vertices $p_1 \to p_2 \to \dots \to p_\ell$, in order.
Then $$\hyperref[def:ra]{R^a} = N(P) \setminus (N(\{p_1, p_\ell\}) \cup W^p).$$

Recall \hyperref[def:wa]{$W^a$} $= \cup_{i=2}^{\ell - 1} F_i^+ \cup L_i^-$, that is, vertices in $N(P) \setminus N(\{p_1,p_{\ell}\})$ whose first neighbor on $P$ is an out-neighbor or whose last neighbor in $V(P)$ is an in-neighbor.
Hence, $\hyperref[def:ra]{R^a} $ consists exactly of the vertices in $N(P) \setminus N(\{p_1,p_{\ell}\})$ whose first neighbor in $V(P)$ is an in-neighbor and whose last neighbor in $V(P)$ is an out-neighbor.
\end{definition}

We bound $\dichi(\hyperref[def:ra]{R^a})$ using a similar technique to the one we used to bound the dichromatic number of \hyperref[def:wa]{$W^a$}.
Recall that we need to bound the dichromatic number of the union of the sets $F_i \backslash W^a$. 
To this end, we prove that there are no arcs between these sets with indices differing by more than three.
We first make the following observation, holding for any shortest directed path. 
\begin{observation}\label{obs:shortcuts}
Let $D$ be an $\hyperlink{def:a4}{\ora{A_4}}$-free oriented graph.
Let $P=p_1 \to p_2 \to \ldots \to p_\ell$ be a shortest directed path from $p_1$ to $p_\ell$ in $D$.
Let $i,j \in \{1,2, \dots, \ell \}$ with $j > i+3$.
Suppose $v, w \in N(P)$ such that $(p_i, v), (w, p_j) \in E(D)$.
Then $(v, w) \not \in E(D)$.
\end{observation}
\begin{proof}
If $(v, w) \in E(D)$ then we may replace the path $p_i \to p_{i+1} \to p_{i+2} \to p_{i+3} \to \dots \to p_j$ with the path $p_i \to v \to w \to p_j$ in $P$ to obtain a shorter directed path from $p_1$ to $p_\ell$, a contradiction.
\end{proof}

The previous observation allows us to prove that some arcs between vertices in $\hyperref[def:ra]{R^a}$ are forbidden.

\begin{observation}\label{obs:RA:dag}
Let $D$ be an $\hyperlink{def:a4}{\ora{A_4}}$-free oriented graph. 
Let $P=p_1 \to p_2 \to \ldots \to p_\ell$ be a shortest directed path in an oriented graph $D$.  Let $(F_1, F_2, \dots, F_\ell)$ be the \hyperref[def:FL]{partition of $N(P)$ by first attachment}.
Then, for any two integers in $i,j \in \{2,3, \dots, \ell -1 \}$ satisfying $i +2 < j$ there is no arc from a vertex in $F_i\setminus W^a$ to a vertex in $F_j \setminus W^a$.
\end{observation}
\begin{proof}
Let $i,j$ be integers $i,j \in \{2,3, \dots, \ell -1 \}$ satisfying $i +2 < j$. 
Suppose $v \in F_i\setminus W^a$ and $w \in F_j \setminus W^a$.
By definition of \hyperref[def:wa]{$W^a$}, for every $r \in N(P) \setminus W^a$, the first neighbor of $r$ in $V(P)$ is an in-neighbor of $r$ and the last neighbor of $r \in N(P)$ is an out-neighbor of $r$.
Hence, there is some $x \in [j+1,\ell-1]$ such that $p_x$ is an out-neighbor of $w$.
Thus, $x > i+3$ and by Observation~\ref{obs:shortcuts}, $(v,w) \not \in E(D)$, see Figure~\ref{fig:shortcutA4}.
\end{proof}

\begin{figure}
    \centering
    \includegraphics[scale=0.15]{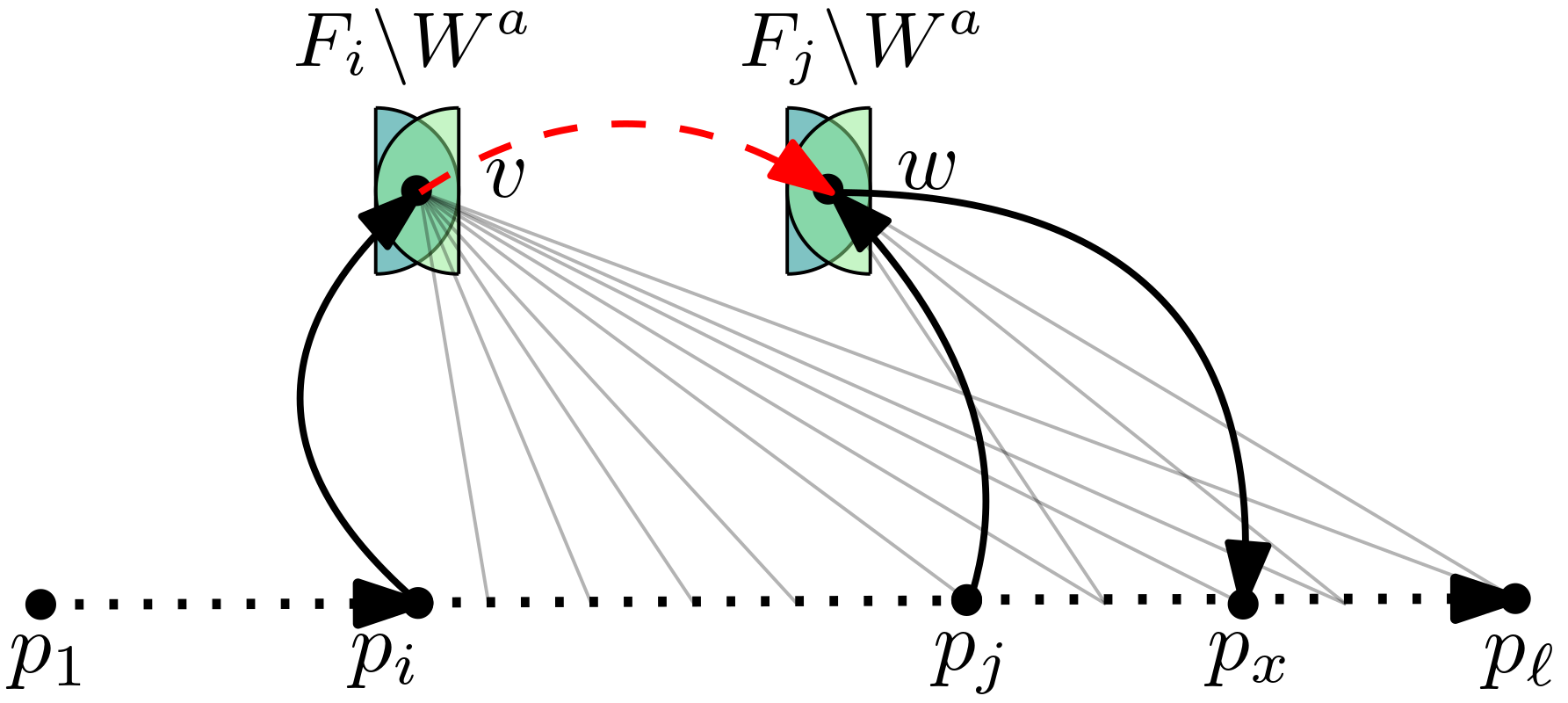}
    \caption{
    A shortest directed path $P$, along with $v \in F_i \setminus W^a$ and $w \in F_j \setminus W^a$ of $\hyperref[def:ra]{R^a}$. Since vertex $w \notin W^a$, it must also belong to some $L_x^+$ with $x > j$, meaning its last neighbor on $P$ is an out-neighbor. Then, since $x>j>i+2$, an arc $(v,w)$ would yield a shorter path from $p_1$ to $p_{\ell}$.}
    \label{fig:shortcutA4}
\end{figure}

With the last observation in hand, we are ready to bound $\dichi(\hyperref[def:ra]{R^a})$, which we do through a similar argument to the case of \hyperref[def:wa]{$W^a$}. The main difference is that here, we use $3$ disjoint pallets of $\gamma$ colors each and choose a color palette for $F_i \backslash W^a$ according to the index of $i$ modulo $3$ (where $i \in \{2, 3, \dots, \ell-1\}$). 
\begin{lemma}\label{lem:a4:nbdP}
Let $D$ be an $\hyperlink{def:a4}{\ora{A_4}}$-free oriented graph, $P$ be a shortest directed path in $D$ between its endpoints, and \hyperref[def:ra]{$R^a$} be defined with respect to $P$.
Let $\gamma$ be an integer satisfying $\dichi(N(v))\leq \gamma$ for each $v \in V(P)$. 
Then $\dichi(\hyperref[def:ra]{R^a})) \leq 3\gamma$.
\end{lemma}
\begin{proof}
Let $P=p_1 \to p_2 \to \ldots \to p_\ell$ be a shortest directed path from $p_1$ to $p_\ell$ in $D$.  Let $(F_1, F_2, \dots, F_\ell)$ be the \hyperref[def:FL]{partition of $N(P)$ by first attachment}.
By defintion $\dichi(F_i) \leq \gamma$ for each $i \in \{1,2, \dots, \ell\}$.
We fix three disjoint sets $S_0, S_1, S_2$ of $\gamma$ colors and dicolor each $F_j \setminus W^a$ for $j \in [2, \ell -1]$ with set $S_i$ where ${i = j \mod 3}$.
By Observation~\ref{obs:RA:dag}, $D[F_2 \cup F_3 \cup \dots \cup F_{\ell -1} \setminus W^a]$ does not contain any monochromatic directed cycle.
Hence, $\dichi(F_2 \cup F_3 \cup \dots \cup F_{\ell-1} \setminus W^p) \leq 3\gamma$, as desired.
\end{proof}

We combine the previous observation with the results of the previous sections to show that $\hyperlink{def:a4}{\ora{A_4}}$-free oriented graphs have a \hyperref[def:niceset]{dipolar} set of bounded dichromatic number.

\begin{lemma}
\label{lem:a4:nicesetbound}
Suppose $D$ is an oriented graph satisfying Scenario~\ref{inductivehypothesis} with $H = \hyperlink{def:a4}{\ora{A_4}}$.
Then $D$ contains a \hyperref[def:niceset]{dipolar} set of dichromatic number at most $(\omega + 7) \cdot \gamma + 2$.
\end{lemma}
\begin{proof}
Let \hyperref[def:closedtournament]{$C = K \cup V(P)$} be a \pathminimizingclosedtournament  closed tournament in $D$.
Let $X$ denote the \hyperlink{def:strongnbd}{strong neighborhood} of $C$.
Then by Lemma~\ref{lem:buildinganiceset}, $N[C \cup X]$ is a \hyperref[def:niceset]{dipolar} set and by Corollary~\ref{corr:Y-dichi}, 
\begin{equation}
\label{koala}
\dichi(N[C \cup X]) \leq \dichi(N(P) \setminus N[K]) + (\omega + 2) \cdot \gamma + 2.
\end{equation}
Let \hyperref[def:wa]{$W^a$} and \hyperref[def:ra]{$R^a$} be defined with respect to $P$.
By definition $p_1, p_\ell \in K$ and so $N(P) \setminus N[K] \subseteq W^a \cup \hyperref[def:ra]{R^a}$.
Thus by combining Lemmas~\ref{lem:Wdichi} and \ref{lem:a4:nbdP}, we obtain 
\begin{equation}
\label{wallaby}
\dichi(N(P) \setminus N[K]) \leq 5 \gamma.
\end{equation}
Then the lemma follows by combining (\ref{koala}) and (\ref{wallaby}).
\end{proof}

\section{Orientations of $P_4$ are $\protect\dichi$-bounding}
\label{sec:gettingthebound}
In this section, we consider an oriented graph $D$ satisfying Scenario~\ref{inductivehypothesis}.
The previous sections show that $D$ has a dipolar set of bounded dichromatic number.
We will use this result and Lemma~\ref{lem:nice-set} to show that oriented graphs not containing some orientation of $P_4$ are $\dichi$-bounded.

\subsection{$D$ contains a dipolar set with bounded dichromatic number}
In the previous sections that if $D$ does not contain some $H \in \{\hyperlink{def:q4}{\ora{Q_4}}, \hyperlink{def:p4}{\ora{P_4}}, \hyperlink{def:a4}{\ora{A_4}} \}$ then $D$ has a \hyperref[def:niceset]{dipolar} set of bounded dichromatic number.
These results can be summarized in the following lemma.

\begin{lemma}
\label{lem:allnicesetbounds}
Let $\omega > 1$ be an integer.
Let $H \in \{\hyperlink{def:q4}{\ora{Q_4}}, \hyperlink{def:p4}{\ora{P_4}}, \hyperlink{def:a4}{\ora{A_4}} \}$.
Let $\gamma$ be the maximum value of $\dichi(D')$ for any $H$-free oriented graph $D'$ with $\omega(D') < \omega$.
Let $D$ be a strongly connected $H$-free oriented graph with clique number $\omega$.
Then, 
\begin{itemize}
    \item If $H = \hyperlink{def:q4}{\ora{Q_4}}$, then $D$ contains a \hyperref[def:niceset]{dipolar} set with dichromatic number at most $(\omega + 3)\cdot \gamma + 2$.
    \item If $H = \hyperlink{def:p4}{\ora{P_4}}$, then $D$ contains a \hyperref[def:niceset]{dipolar} set with dichromatic number at most $(\omega+ 6)\cdot \gamma +2$.
    \item If $H = \hyperlink{def:a4}{\ora{A_4}}$, then $D$ contains a \hyperref[def:niceset]{dipolar} set with dichromatic number at most $(\omega + 7) \cdot \gamma +2$.
\end{itemize}
\end{lemma}
\begin{proof}
The result for $H = \hyperlink{def:q4}{\ora{Q_4}}, \hyperlink{def:p4}{\ora{P_4}}, \hyperlink{def:a4}{\ora{A_4}}$ is given in Lemmas~\ref{lem:q4:nicesetbound}, \ref{lem:RP-dichi} and \ref{lem:a4:nicesetbound}, respectively.
\end{proof}

\subsection{Computing the $\protect\dichi$-binding function}

We will show that an element from the following family of functions is a $\dichi$-binding function for any class of oriented graphs forbidding a particular orientation of $P_4$.

\begin{definition}
\label{def:fc}
For any integer $c \geq 3$ we let 
$$\textcolor{bordeaux}{f_c}(x) = 2^{x}(x+c)! + \sum_{i=0}^{x} \frac{2^{i+2}(x+c)!}{(x+c-i)!}$$ for any non-negative integer $x$.
\end{definition}

We will need that $\hyperref[def:fc]{f_c}$ satisfies the following recursive properties in order to show that for some $c$ the function $\hyperref[def:fc]{f_c}$ is $\chi$-bounding for any class of oriented graph forbidding a particular orientation of $P_4$.

\begin{observation}\label{obs:reccurencerelation}
Let $c \geq 3$.
Then:
\begin{itemize}
    \item $\hyperref[def:fc]{f_c}(x) = 2(x+c)\hyperref[def:fc]{f_c}(x-1) + 4$ for any integer $x \geq 2$, and
    \item $\hyperref[def:fc]{f_c}(1) > 1$,
    \item $\hyperref[def:fc]{f_c}(x) \leq(x+c)^{x+c+1.5}$ %
      for any integer $x \geq 1$.
\end{itemize}
\end{observation}

\begin{proof}
By definition since $c \geq 3$ we obtain that  $\hyperref[def:fc]{f_c}(1) > 2 c! > 1$.
Hence, the second bullet holds and we will now prove the first bullet.
Let $x \geq 2$ be an integer.
Then,
$$\hyperref[def:fc]{f_c}(x) = 2(x+c)\left( 2^{x-1}(x+c-1)! + \frac{1}{2(x+c)}\sum_{i=0}^x \frac{2^{i+2}(x+c)!}{(x+c-i)!} \right).$$
By definition,
$$\sum_{i=0}^x \frac{2^{i+2}(x+c)!}{(x+c-i)!} = 2(x+c)\left(\sum_{i=1}^{x}\frac{2^{i+1}(x+c-1)!}{(x+c-i)!}\right) + 4 = 2(x+c)\left(\sum_{i=0}^{x-1}\frac{2^{i+2}(x+c-1)!}{(x+c-i-1)!}\right) + 4.$$
Thus, by combining the previous two equations we obtain that $\hyperref[def:fc]{f_c}(x) = 2(x+c)\hyperref[def:fc]{f_c}(x-1) + 4$.
This proves the first bullet.

We will complete the proof by showing the third bullet holds. Let $x \geq 1$.
By definition, $$\sum_{i=0}^{x} \frac{2^{i+2}(x+c)!}{(x+c-i)!} = 2^2 + 2^3(x+c) + 2^4(x+c)(x+c-1) + \dots + 2^{x+2}\frac{(x+c)!}{c!}.$$
Since $c! > 4$ every $i \in \{0, \dots, x\}$ satisfies $\frac{2^{i+2}(x+c)!}{(x+c-i)!} \leq 2^x(x+c)!$.
Hence, we obtain 
$$
\sum_{i=0}^{x} \frac{2^{i+2}(x+c)!}{(x+c-i)!} \leq 2^x (x+1)(x+c)!.
$$

\begin{equation}\label{eqn:fc:exp1}
    \hyperref[def:fc]{f_c}(x) < 2^x(x+c)!+2^x (x+1)(x+c)! \leq (x+2) 2^{x}(x+c)!.
\end{equation}%
We will use the following well-known equation called \emph{Stirling's Formula} to complete the proof.
\stmt{\label{eqn:stirling}
Every $n \geq 1$ satisfies
$$\left(n!<\sqrt{2\pi n}\left(\frac{n}{e}\right)^n e^{\frac{1}{12n}}\right). $$
}
Since $c \geq 3$ we obtain the following by combining (\ref{eqn:fc:exp1}) and (\ref{eqn:stirling}).
\begin{align*}
  (x+2) 2^{x}(x+c)!<&
  (x+2) 2^{x}\sqrt{2\pi (x+c)}\left(\frac{(x+c)}{e}\right)^{x+c} e^{\frac{1}{12(x+c)}} %
   <%
  (x+c)^{x+c+1.5}.
\end{align*}
This proves the third bullet.
\end{proof}

\subsection{$\protect\dichi$-boundedness}

We are now ready to prove the following more precise version of our main result, Theorem~\ref{thm:main}.
\begin{theorem}\label{thm:q4}%
Let $H$ be an orientation of $P_4$, then $H$-free graphs are $\dichi$-bounded.
Specifically,
\begin{itemize}
    \item If $D$ is $\hyperlink{def:q4}{\ora{Q_4}}$-free or $\ora{Q'_4}$-free, then $\dichi(D) \leq (\omega(D)+3)^{\omega(D)+4.5}$,
    \item If $D$ is $\hyperlink{def:p4}{\ora{P_4}}$-free, then $\dichi(D) \leq (\omega(D)+6)^{\omega(D)+7.5}$, and
    \item If $D$ is $\hyperlink{def:a4}{\ora{A_4}}$-free, then $\dichi(D) \leq (\omega(D)+7)^{\omega(D)+8.5}$.
\end{itemize}
\end{theorem}

\begin{proof}
Let $H$ be an orientation of $P_4$
Note $\hyperlink{def:q4}{\ora{Q_4}}$ can be obtained from $\hyperlink{def:q4}{\ora{Q_4}}'$ by reversing the orientation of every edge.
Hence the theorem holds for $\hyperlink{def:q4}{\ora{Q_4}}$ if and only if it holds for $\ora{Q_4'}$.
Thus may assume $H \in \{\hyperlink{def:q4}{\ora{Q_4}}, \hyperlink{def:p4}{\ora{P_4}}, \hyperlink{def:a4}{\ora{A_4}} \}$.

We let $c = 3$ if $H = \hyperlink{def:q4}{\ora{Q_4}}$, $c = 6$ if $H = \hyperlink{def:p4}{\ora{P_4}}$ and $c = 7$ if $H = \hyperlink{def:a4}{\ora{A_4}}$.
Then by the third bullet of Observation~\ref{obs:reccurencerelation}, it is enough to show that the class of $H$-free oriented graphs is $\dichi$-bounded by $\hyperref[def:fc]{f_c}$.

We have $\hyperref[def:fc]{f_c}(1) > 1$ by the first bullet of Observation~\ref{obs:reccurencerelation}, so the statement holds for oriented graphs with no arcs.
We complete the proof by induction on the clique number.
Let $\omega > 1$ be an integer.
Suppose every $H$-free oriented graph $D'$ with clique number less than $\omega$ satisfies $\dichi(D') \leq f(\omega(D'))$.
Let $D$ be an $H$-free oriented graph with clique number equal to $\omega$.
We will show $\dichi(D) \leq \hyperref[def:fc]{f_c}(\omega)$.
We may assume by induction on the number of vertices that $D$ is strongly connected.

By Lemma~\ref{lem:allnicesetbounds}, $D$ has a \hyperref[def:niceset]{dipolar} set $S$ with  $\dichi(S) \leq (\omega + c)\cdot \hyperref[def:fc]{f_c}(\omega-1) + 2$.
Then by Lemma~\ref{lem:nice-set}, $$\dichi(D) \leq 2 \cdot \dichi(S) \leq 2(\omega + c)\cdot \hyperref[def:fc]{f_c}(\omega-1) + 4. $$
Since $\omega \geq 2$ this implies $\dichi(D) \leq \hyperref[def:fc]{f_c}(\omega)$ by the second bullet of Observation~\ref{obs:reccurencerelation}. 
This completes the proof.
\end{proof}

\section{Conclusion}
\label{sec:conclusion}

Our result is an initial step towards resolving the \hyperref[con:AlboukerGyarfasSumner4Digraphs2020]{ACN $\dichi$-boundedness conjecture} for orientation of paths in general. 
However, we think we are still far from this result.
Our construction of a \hyperref[def:niceset]{dipolar} set with bounded chromatic number relies heavily on the length of $P_4$ and we do not expect that our techniques can be directly extended to show that any oriented $P_t$ for $t \geq 5$ is $\dichi$-bounding.
It would already be interesting to hear the answer to the easier question: Does there exists an integer $t \geq 5$ and an orientation $H$ of $P_t$ such that the class of oriented graphs forbidding $H$ and all tournaments of size $3$ has unbounded dichromatic number?

Recall that the classes of $\hyperlink{def:q4}{\ora{Q_4}}$-free oriented graphs and $\hyperlink{def:q4}{\ora{Q_4}'}$-free oriented graphs were already shown to be $\chi$-bounded in \cite{chudnovsky2019orientations}.
The $\chi$-binding function $f'$ for these two classes from~\cite{chudnovsky2019orientations} is defined using recurrence \[f'(x)\coloneqq 2(3 f'(x-1))^5\] which leads to a double-exponential bound on $\chi$, and cannot guarantee a better bound on $\dichi$.
In this paper, Theorem~\ref{thm:main} provides an improved $\dichi$-binding function when any orientation of $P_4$ is forbidden. %
It would interest us to know of any improvements to the $\dichi$ function.
In particular, we would like to know whether any orientation of $P_4$ is \emph{polynomially} $\dichi$-bounding. In other words, is there some oriented $P_4$ so that the class of oriented graphs forbidding it has a \emph{polynomial} $\dichi$-binding function?

\section*{Acknowledgements}
Much of the work in this paper was done during the Sparse Graphs Coalition online workshop on directed minors and digraph structure theory held on March 4--8, 2022.
We would like to thank Pierre Aboulker for suggesting this problem to us at the workshop and for his help with the bibliography.
We also would like to thank Raphael Steiner for informing us of his prior related work \cite{steiner2021}.
Finally, we would like to thank the workshop participants and organizers, 
Jean-Sébastien Sereni, Raphael Steiner, and Sebastian Wiederrecht, for creating a friendly research environment.

\bibliography{bibliography_dichi.bib}

\end{document}